\documentclass[twoside,leqno,11pt]{article}
\usepackage{revstat}
\begin{document}
\title{GAMMA KERNEL ESTIMATION OF THE DENSITY DERIVATIVE ON THE POSITIVE SEMI-AXIS BY DEPENDENT DATA}
\renewcommand{\titleheading}
             {Gamma kernel estimation}  
\author{\authoraddress{L. A. Markovich}
                      {Institute of Control Sciences,
         Russian Academy of Sciences,\\
Moscow, Russia.
                       \ (kimo1@mail.ru)}
}
\renewcommand{\authorheading}
             {L. A. Markovich}  

\maketitle

\begin{abstract}
We estimate the derivative of a probability density function defined on
 $[0,\infty)$. For this purpose, we choose the class of kernel estimators with asymmetric
gamma kernel functions. The use of gamma kernels is fruitful due to the fact that
they are nonnegative, change their shape depending on the position
on the semi-axis and possess good boundary properties for a wide class of densities. We find an optimal bandwidth
of the kernel as a minimum of the mean integrated squared error by dependent data with strong mixing. This bandwidth
differs from that proposed for the gamma kernel density estimation.
To this end, we derive the covariance of derivatives of the density and deduce its upper bound. Finally,
the obtained results  are applied to the case of a first-order autoregressive process with strong mixing. The accuracy of the
estimates is checked by a simulation study. The comparison of the proposed estimates based on independent and dependent data
is provided.
\end{abstract}


\begin{ams}
60G35, 60A05.
\end{ams}

\mainpaper  

\section{INTRODUCTION}  \label{Sec1}              
Kernel density estimation is a non-parametric method to estimate a probability density function (pdf) $f(x)$. It was originally studied in
\cite{Parzen}, \cite{Rosenblatt} for symmetric kernels and univariate independent identically distributed (i.i.d) data. When the support of the underlying pdf is unbounded, this approach performs well. If the pdf has a support on $[0,\infty)$, the use of classical estimation methods with symmetric
kernels yield a large bias on the zero boundary and leads to a bad quality of the estimates \cite{WandJones}. This is due to the fact that symmetric kernel estimators assign nonzero weight at the interval $(-\infty,0]$.  There are several methods to
reduce  the boundary bias effect, for example, the data reflection \cite{Schuster}, boundary kernels \cite{Muller}, the hybrid method \cite{HallWehrly}, the local linear estimator \cite{LejeuneSarda}, \cite{Jones} among others.
Another approach is to use  asymmetric kernels. In case of univariate nonnegative i.i.d random variables (r.v.s), the  pdf estimators with gamma kernels were proposed in \cite{Chen:20}. In \cite{TaufikBouezmarnia:Rom2} the gamma-kernel estimator was developed for univariate dependent data. The gamma kernel is nonnegative and it changes its shape depending on the position on the semi-axis. Estimators constructed with gamma kernels have no boundary bias if $f''(0)=0$ holds, i.e when the underlying density $f(x)$ has a shoulder at $x = 0$ (see formula (4.3) in \cite{Zhang}). This shoulder property is fulfilled particularly for a wide exponential class of pdfs which satisfy important integral condition
\begin{equation}\int_{0}^{\infty}x^{-1/2}f(x)dx<\infty\label{0}
\end{equation}
assumed in \cite{Chen:20}. In \cite{Zhang} the half normal and standard exponential pdfs are considered as examples such that the boundary kernel $K_c(t)$ (p. 553 in \cite{Zhang}) gives the better estimate  than the gamma-kernel estimator considered in \cite{Chen:20}.
At the same time, the exponential distribution does not satisfy both the shoulder condition and the condition \eqref{0}. The half normal density satisfies the shoulder condition, but it does not satisfy \eqref{0}. Since \eqref{0} is not valid for the latter pdfs, such comparison is not appropriate.
\par Alternative asymmetrical kernel estimators like inverse Gaussian and reciprocal inverse Gaussian estimators were studied in \cite{Scailet}. The comparison of these asymmetric kernels with the gamma kernel is given in \cite{BouSca}.
\par Along with the density estimation it is often necessary  to estimate the derivative of a pdf.
Derivative estimation is important in the exploration of structures in curves, comparison of regression curves, analysis of human growth data, mean shift clustering or  hypothesis testing. The estimation of the density derivative is required to estimate the logarithmic derivative of the density function.
The latter has a practical importance in finance, actuary mathematics, climatology and signal processing.
However, the problem of the density derivative estimation has received less attention.
It is due to a significant increasing complexity of calculations, especially for the multivariate case.
The boundary bias problem  for the multivariate pdf becomes more solid \cite{TaufikBouezmarnia:Rom}.
The pioneering papers devoted to univariate symmetrical kernel density derivative estimation are \cite{Bhattacharya}, \cite{Schuster2}.
\par The paper does not focus on the boundary performance but on finding of the optimal bandwidth that is appropriate for the pdf derivative estimation in case of dependent data satisfying a strong mixing condition.
In \cite{WandJones} an optimal mean integrated squared error (MISE) of the kernel estimate of the first derivative of order $n^{-\frac{4}{7}}$ was indicated.
This corresponds to the optimal bandwidth of order $n^{-\frac{1}{7}}$ for symmetrical kernels.
The estimation of the univariate density derivative using a gamma kernel estimator  by independent data was proposed in \cite{DobrovidovMarkovich:13a}, \cite{DobrovidovMarkovich:13b}.
This allows us to achieve the optimal MISE of the same order  $n^{-4/7}$ with a bandwidth of order $n^{-\frac{2}{7}}$.
\subsection{Contributions of this paper}
It is shown that in the case of dependent data, assuming strong mixing, we can estimate the derivative of the pdf using the same technique that has been applied for independent data in \cite{DobrovidovMarkovich:13a}. Lemma \ref{Lem1}, Section \ref{SectLem1} contains  the upper bound of the covariance.
The mathematical technic applied for the derivative estimation is similar to one applied for the pdf. However, formulas became much more complicated,
particulary because one has to deal with the special Digamma function that includes the bandwidth $b$. Thus, one has to pick out the order by $b$ from complicated expressions containing logarithms and the special function.
In Section \ref{SectMISE} we find the optimal bandwidth $b\sim n^{-2/7}$ which is different from the optimal bandwidth $b_{2}^{\ast}\sim n^{-2/5}$ proposed for the pdf estimation (see \cite{Chen:20}, p. 476).
In Fig.~\ref{Fig0} it is shown that the use of $b_{2}^{\ast}$ to estimate the pdf derivative leads to a bad quality (for simplicity the i.i.d data were taken). We prove that the optimal $MISE$ of the pdf derivative has the same rate of convergence  to the true pdf
derivative as for the independent case, namely $O(n^{-4/7})$.
\begin{figure}[Ht]
\begin{center}
\begin{minipage}[Ht]{0.60\linewidth}
\includegraphics[width=1\linewidth]{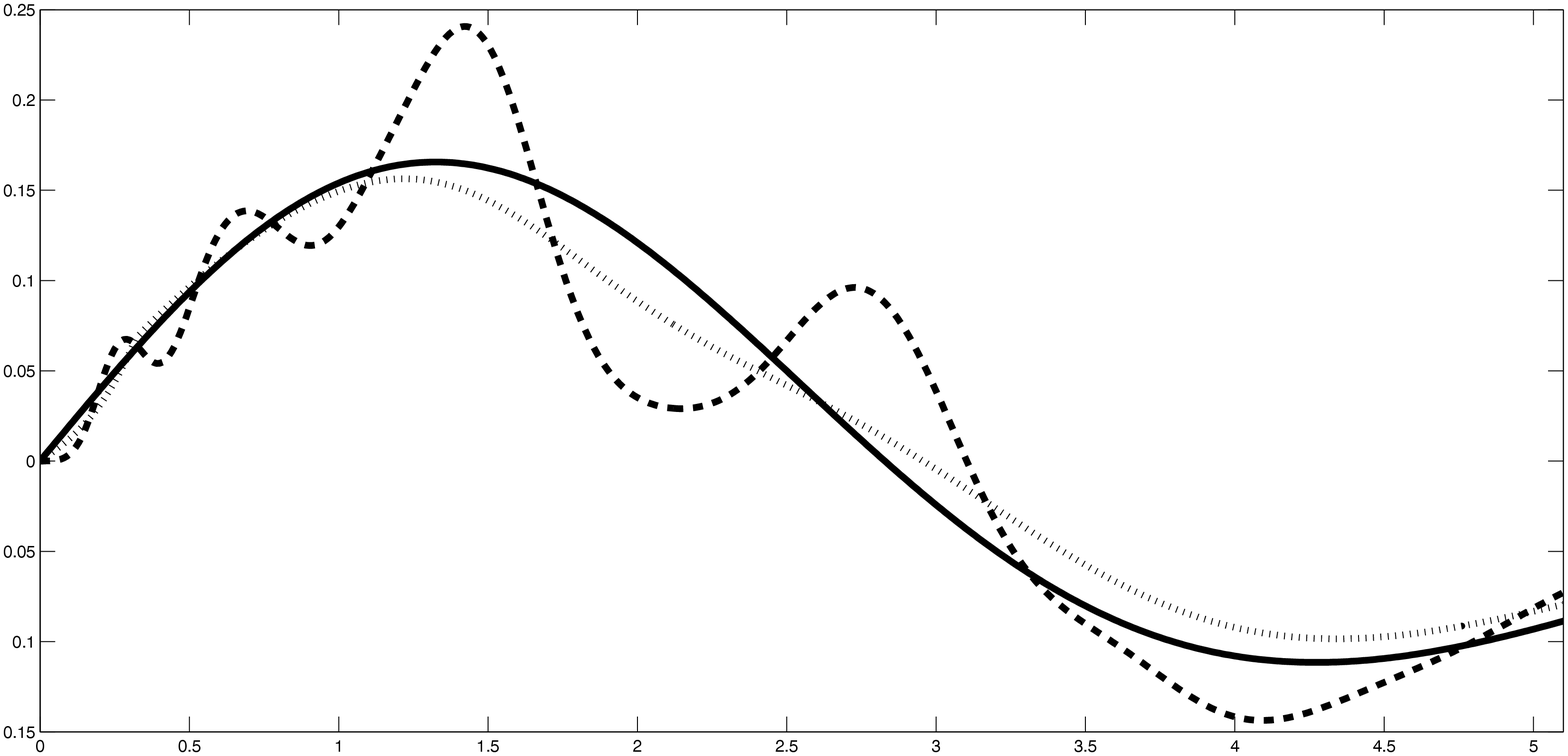}
\vspace{-4mm}
\end{minipage}
 \caption{Nonparametric gamma-kernel estimation of Maxwell density derivative
function for sample size n=2000. The pdf derivative (solid line), the estimate with $b$
(dotted gray line), the estimate with $b_{2}^{\ast}$ (dashed  line).\label{Fig0}}
\end{center}
\end{figure}
We  show in Section \ref{SectAR} that for the strong mixing autoregressive process of the first order (AR(1)) all results are valid without additional conditions.
In Section \ref{Sim} a simulation study for i.i.d and dependent samples is performed. The flexibility of the gamma kernel allows us to fit accurately the multi-modal pdf derivatives.
\subsection{Practical motivation}
In practice it is often necessary to deal with sequences of observations that are derived from stationary processes satisfying the strong mixing condition.
As an example of such processes one can take  autoregressive processes like in Section \ref{SectAR}.
Along with the evaluation of the density function and its derivative
by dependent samples, the estimation of the logarithmic derivative of the density is an actual problem. The logarithmic pdf derivative  is the ratio of the derivative of the pdf to the pdf itself. The pdf derivative  estimation is necessary for an optimal filtering in the signal processing and control of nonlinear processes where only the exponential pdf class is used, \cite{Dobrovidov:12}.  
Moreover, the pdf derivative gives information about the slope of the pdf curve, its local extremes, significant features in data and it is useful in regression analysis \cite{Brabanter}. The pdf derivative also plays a key role in clustering via mode seeking \cite{Sasaki}.
\subsection{Theoretical background}
Let $\{X_i;i=1,2,\ldots\}$ be a strongly stationary sequence with an unknown probability density function  $f(x)$,
which is defined on  $x\in[0,\infty)$. We assume  that the sequence $\{X_i\}$ is $\alpha-$mixing with coefficient
\begin{equation*}
    \alpha(i)  = \sup\limits_{k}\sup\limits_{\genfrac{}{}{0pt}{}{A\in \mathcal{F}_1^k(X)}{B\in \mathcal{F}_{k+i}^\infty(X)}}|P(A\cap B)-P(A)P(B)|.\label{1}
\end{equation*}
Here, $\mathcal{F}_{i}^k(X)$ is the $\sigma$-field of events generated by $\{X_j, i\leq j\leq k\}$ and $\alpha(i)\rightarrow0$ as $i\rightarrow\infty$.
For these sequences we will use a notation $\{X_j\}_{j\geq1}\in\mathcal{S}(\alpha)$.
Let  $f_i(x,y)$ be a joint density of $X_1$ and $X_{1+i}$, $i=1,2,\ldots$.
\par Our objective is to estimate the derivative $f'(x)$ by a known sequence of observations $\{X_i\}$. We use the non-symmetric gamma kernel estimator that was defined in \cite{Chen:20} by the formula
\begin{equation}
\widehat{f}_n(x) = \frac{1}{n}\sum\limits_{i=1}^{n}K_{\rho_b(x),b}(X_i)\label{2}.
\end{equation}
Here
\begin{equation}\label{Krho}K_{\rho_b(x),b}(t)=\frac{t^{\rho_b(x)-1}\exp(-t/b)}{b^{\rho_b(x)}\Gamma(\rho_b(x))}
\end{equation}
is the kernel function, $b$ is a smoothing parameter (bandwidth) such that $b\rightarrow 0$ as $n\rightarrow\infty$,
$\Gamma(\cdot)$ is a standard gamma function and
\begin{eqnarray}\label{rho}
\rho_b(x)&=& \left\{
\begin{array}{ll}
\rho_1(x) = x/b, &   \mbox{if}\qquad x\geq 2b,
\\
\rho_2(x) =\left(x/(2b)\right)^2+1, & \mbox{if}\qquad x\in
[0,2b).
\end{array}
\right.
\end{eqnarray}
The use of gamma kernels is due to the fact that they are nonnegative, change their shape depending on the position
on the semi-axis and possess better boundary bias than symmetrical kernels. The boundary
bias becomes larger for multivariate densities. Hence, to overcome this problem  the gamma kernels were applied in \cite{TaufikBouezmarnia:Rom}.
Earlier the gamma kernels were only used for the density estimation of identically distributed sequences in \cite{TaufikBouezmarnia:Rom}, \cite{Chen:20}  and for stationary sequences in \cite{TaufikBouezmarnia:Rom2}.
\par To our best knowledge, the gamma kernels have been applied to the density derivative estimation at first time in \cite{DobrovidovMarkovich:13a}.
 In this paper the derivative $f'(x)$ was estimated under the assumption that $\{X_1,X_2,\ldots,X_n\}$ are i.i.d random variables as derivative of \eqref{2}. This implies
 that
\begin{eqnarray}\label{f'(x)}
\hat{f}'_n(x) &=& \frac{1}{n}\sum_{i=1}^{n}K'_{\rho_b(x),b}(X_i)
\end{eqnarray}
holds, where
\begin{eqnarray}\label{K'}
K'_{\rho_b(x),b}(t)&=&\left\{
\begin{array}{ll}
K'_{\rho_1(x),b}(t)=\frac{1}{b}K_{\rho_1(x),b}(t)L_1(t),
& \mbox{if}\quad x\geq 2b,\\
K'_{\rho_2(x),b}(t)=\frac{x}{2b^2}K_{\rho_2(x),b}(t)L_2(t),&
\mbox{if}\quad x\in [0,2b),
\end{array}
\right.
\end{eqnarray}
is the derivative of $K_{\rho(x),b}(t)$,
\begin{eqnarray}\label{L}
L_1(t)&=&L_1(t,x)= \ln t - \ln b - \Psi(\rho_1(x)),\\
L_2(t)&=&L_2(t,x)= \ln t - \ln b - \Psi(\rho_2(x)),\nonumber
\end{eqnarray}
Here $\Psi(x)$  denotes  the Digamma function (the logarithmic derivative of the gamma
function).
 The unknown smoothing parameter $b$ was obtained as the minimum of the mean integrated squared error ($MISE$) which, as
known, is equal to
\begin{eqnarray}\label{3}
MISE(\hat{f}'_n(x))&=&
\mathsf E\int\limits_0^\infty(f'(x)-\hat{f}'_n(x))^2dx. \notag
\end{eqnarray}
\begin{remark}
The latter integral can be splitted into two integrals $\int_{0}^{2b}$ and $\int_{2b}^{\infty}$.
In the case when $x\geq 2b$  the integral $\int_{0}^{2b}$ tends to zero when $b\rightarrow0$. Hence, we omit the consideration of this integral in contrast to
 \cite{Zhang}. The first integral has the same order by $b$ as the second one, thus it cannot affect on the selection of the optimal bandwidth.
 \end{remark}
The following theorem has been proved.
\begin{theorem}\label{Thb}
\cite{DobrovidovMarkovich:13a}
\par If $b\rightarrow 0$ and \ $
nb^{3/2}\rightarrow \infty$ as $n\rightarrow \infty$, the  integrals
 \begin{eqnarray*}
\int\limits_{0}^\infty
P(x)dx, \quad\int\limits_{0}^\infty
x^{-3/2} f(x)dx
\end{eqnarray*}
are finite and $\int\limits_{0}^\infty P(x)dx\neq0$, then the leading term of a MISE expansion of the
density derivative estimate  $\hat{f}'(x)$ is equal to
 \begin{eqnarray} \label{5}&& MISE(\hat f'_n(x))=\frac{b^2}{16}\int_{0}^\infty
 P(x)dx\\
 &+&\int_0^\infty \frac{n^{-1}b^{-3/2}x^{-3/2}}{4\sqrt{\pi}}\left(f(x)+b\left(\frac{f(x)}{2x}-\frac{f'(x)}{2}\right)\right)dx
+o(b^2 + n^{-1}(b^{-3/2})).\nonumber
\end{eqnarray}
where
\begin{eqnarray*}
P(x)&=&\left(\frac{f(x)}{3x^2}+f''(x)\right)^2.
\end{eqnarray*}
\end{theorem}
Taking the derivative of \eqref{5} in $b$ leads to equation
\begin{eqnarray}\label{b^8}
&&\frac{b}{8}\int_{0}^\infty
\left(\frac{f(x)}{3x^2}+f''(x)\right)^2dx-
\frac{3n^{-1}b^{-\frac{5}{2}}}{8\sqrt{\pi}}\int_0^\infty x^{-\frac{3}{2}}f(x)dx\\\nonumber
&+&\frac{n^{-1}b^{-\frac{3}{2}}}{16\sqrt{\pi}}\int_0^\infty x^{-\frac{3}{2}}\left(\frac{f(x)}{x}-f'(x)\right)dx=0.
\end{eqnarray}
Neglecting the term with $b^{-3/2}$ as compared to the term
$b^{-5/2}$, the equation becomes simpler
and  its solution is  equal to the optimal global bandwidth
\begin{eqnarray}\label{6}
 b_0 = \left(\frac{3\int_0^\infty x^{-3/2}f(x)dx}{\sqrt{\pi}\int_{0}^\infty
\left(\frac{f(x)}{3x^2}+f''(x)\right)^2dx}\right)^{2/7}n^{-2/7}.
\end{eqnarray}
The substitution of $b_0$ into \eqref{5} yields an optimal $MISE$ with the rate of convergence $O(n^{-\frac{4}{7}})$.
The unknown density and its second derivative in \eqref{6} were estimated by the rule of thumb method  \cite{DobrovidovMarkovich:13b}.
\par In \cite{WandJones}, p. 49, it was indicated an optimal $MISE$ of the first derivative kernel estimate $n^{-\frac{4}{7}}$ with the bandwidth of order $n^{-\frac{1}{7}}$ for symmetrical kernels. Nevertheless, our procedure achieves the same order  $n^{-4/7}$ with a bandwidth of order $n^{-\frac{2}{7}}$. Moreover, our advantage concerns the reduction of the bias of the density derivative at the zero boundary by means of asymmetric kernels.
Gamma kernels allow us to avoid boundary transformations which is especially important for multivariate cases.
\par Further results presented in Section \ref{SectMISE} will be based on Theorem \ref{Thb}.
\section{Main Results}
\subsection{Estimation of the density derivative by dependent data}\label{SectLem1}
Here, we estimate the density derivative by means of the kernel estimator \eqref{f'(x)} by dependent data.
Thus, its mean squared error is determined as
\begin{eqnarray}\label{MSE}MSE(\widehat{f'}_n(x))= (Bias(\widehat{f'}_n(x)))^2+var(\widehat{f'}_n(x)),
\end{eqnarray}
where, due to the stationarity of the process $X_i$, the variance is given by
\begin{eqnarray*}\mbox{var}(\widehat{f'}_n(x))&=&\mbox{var}\left(\frac{1}{n}\sum\limits_{i=1}^{n}K'_b(X_i)\right)=\frac{1}{n^2}
\mbox{var}\left(\sum\limits_{i=1}^{n}K'_b(X_i)\right)\\
&=&\frac{1}{n^2}\left(\sum\limits_{i=1}^{n}\mbox{var}(K'_b(X_i))+2\sum\limits_{1\leq i<j\leq n}\mbox{cov}(K'_b(X_i),K'_b(X_j))\right)\\
&=&\frac{1}{n}\mbox{var}(K'_b(X_i))+\frac{2}{n^2}\sum\limits_{1\leq i<j\leq n}\mbox{cov}(K'_b(X_i),K'_b(X_j))\\
&=&\frac{1}{n}\mbox{var}(K'_b(X_i))+\frac{2}{n}\sum\limits_{i=1}^{n-1}\left(1-\frac{i}{n}\right)\mbox{cov}(K'_b(X_1),K'_b(X_{1+i}))\\
&=&V(x)+C(x).
\end{eqnarray*}
 For simplicity we use here and further the notation $K'_{\rho_b(x),b}(t)=K'_{b}(t)$ in \eqref{f'(x)}.
\par Thus, \eqref{MSE} can be written as
\begin{eqnarray}\label{900}MSE(\widehat{f'}(x))= B(x)^2+V(x)+C(x),
\end{eqnarray}
where
\begin{eqnarray*}B(x)&=&Bias(\widehat{f'}_n(x)).\end{eqnarray*}
The bias of the estimate does not change, but the variance contains a covariance.
The next lemma is devoted to its finding.
\begin{lemma}\label{Lem1} Let
\begin{enumerate}
  \item $\{X_j\}_{j\geq1}\in\mathcal{S}(\alpha)$ and $\int\limits_1^\infty \alpha(\tau)^\upsilon d\tau<\infty,\quad 0<\upsilon<1$ hold,
  \item $f(x)$ be a twice continuously differentiable function,
  \item $b\rightarrow 0$ and $nb^{-(\upsilon+1)/2}\rightarrow \infty$ as $n\rightarrow\infty$.
\end{enumerate}
Then the covariance $C(x)$ is bounded by
\begin{eqnarray}\label{4}&&|C(x)|=\Bigg|\frac{2}{n}\sum\limits_{i=1}^{n-1}\left(1-\frac{i}{n}\right)cov(K'_{\rho_b(x),b}(X_1),K'_{\rho_b(x),b}(X_{1+i}))\Bigg|\\\nonumber
&\leq&\Bigg(2^{-\frac{\upsilon+3}{2}}\pi^{\frac{1-\upsilon}{2}}x^{-\frac{\upsilon+5}{2}}\frac{b^{-\frac{\upsilon+1}{2}}}{n}
\Bigg(b^2C_2(\upsilon,x)+bC_1(\upsilon,x)+C_3(\upsilon,x)\Bigg)^{1-\upsilon}\\\nonumber
&+&o(b^2)\Bigg)\int\limits_{1}^{\infty}\alpha(\tau)^{\upsilon}d\tau,
\end{eqnarray}
where $K'_{\rho_b(x)}$ is defined by \eqref{K'} and $C_1(\upsilon,x)$, $C_2(\upsilon,x)$ and $C_1(\upsilon,x)$ are given by \eqref{C123}.
\end{lemma}
A similar lemma was proved in \cite{Dobrovidov:12} for symmetrical kernels and not strictly positive $x$.
\subsection{Mean integrated squared error of $\hat{f'}_n(x)$}\label{SectMISE}
Using the upper bound \eqref{4} we can obtain the upper bound of the MISE and find the expression of the optimal bandwidth $b$ as the minimum of the latter.
 \begin{theorem}\label{thm}If the conditions of Theorem \ref{Thb} and Lemma \ref{Lem1} hold,
then the 
MISE expansion for the
estimate  $\hat{f'}_n(x)$ of the density derivative is equal to
\begin{eqnarray}\label{10}\nonumber &&MISE(f'(x))\leq\int\limits_0^\infty \frac{n^{-1}b^{-\frac{3}{2}}x^{-\frac{3}{2}}}{4\sqrt{\pi}}\left(f(x)+\frac{b}{2}\left(\frac{f(x)}{x}-f'(x)\right)\right)dx
\\
&+& \int\limits_{0}^{\infty}\Bigg(2^{-\frac{\upsilon+3}{2}}\pi^{\frac{1-\upsilon}{2}}x^{-\frac{\upsilon+5}{2}}\frac{b^{-\frac{\upsilon+1}{2}}}{n}
C_3(\upsilon,x)^{1-\upsilon}\Bigg)\int\limits_{1}^{\infty}\alpha(\tau)^{\upsilon}d\tau dx\nonumber \\
&+&\frac{b^2}{16}\int\limits_{0}^\infty P(x)dx+o(b^2 + n^{-1}(b^{-\frac{3}{2}})).
\end{eqnarray}
and the optimal bandwidth is $b_{opt}=o(n^{-2/7})$  and  the $MISE_{opt}=O(n^{-4/7}).$
 \end{theorem}
 \begin{remark}It is evident from the formula \eqref{10} that the term responsible for the covariance has the order $\frac{b^{-\frac{\upsilon+1}{2}}}{n}$, $0<\upsilon<1$.
 Thus, it does not influence the order of MISE irrespective of the mixing coefficient $\alpha(\tau)$.
 \end{remark}
 The proof is given in Appendix \ref{ap2}.
\subsection{Example of a strong mixing process}\label{SectAR}
We use the first-order autoregressive process as an example of a process that satisfies Theorem \ref{Thb}.
$X_i$ determines a first-order autoregressive (AR(1)) process with the innovation r.v. $\epsilon_0$ and the autoregressive parameter $\rho\in(-1,1)$ if
\begin{eqnarray}&&\label{15} X_i=\rho X_{i-1}+\epsilon_i,\quad i=\ldots-1, 0, 1,\ldots,
\end{eqnarray}
holds and $\epsilon_i$ is a sequence of i.i.d r.v.s
Let AR(1) process \eqref{15} be strong mixing with mixing numbers $\alpha(\tau)$, $\tau=1,2,\ldots$
\begin{eqnarray}\label{alpha}
\alpha(\tau)\leq\widetilde{\alpha}(\tau)&\equiv& \left\{
\begin{array}{ll}
2(C+1)\mathsf E|X_i|^\nu|\rho^\nu|^\tau, &   \mbox{if}\qquad \tau\geq\tau_0,
\\
1, & \mbox{if}\qquad 1\leq\tau<\tau_0,
\end{array}
\right.
\end{eqnarray}
where $\nu=\min\{p,q,1\}$ and  $ p>0, q>0, C>0, \tau_0>0$ hold.
In \cite{Donald:83} it was proved that with some conditions AR(1) is a strongly mixing process.
\par In Appendix \ref{ap3} we prove the following lemma.
\begin{lemma}\label{lem3}
Under the conditions \eqref{alpha} the AR(1) process \eqref{15} satisfies Lemma \ref{Lem1} and Theorem \ref{thm}.
\end{lemma}
\section{Simulation results}\label{Sim}
To investigate the performance of the gamma-kernel estimator we select the following positive defined pdfs: the
 Maxwell ($\sigma=2$), the Weibull ($a =1, b=4$) and the Gamma ($\alpha=2.43,\beta=1$) pdf,
 \begin{eqnarray*}
f_M(x)&=&\frac{\sqrt{2}x^2\exp(-x^2/2\sigma^2)}{\sigma^3\sqrt{\pi}},\\
f_W(x)&=&sx^{s-1}\exp(-x^s),\\
f_G(x)&=&\frac{x^{\alpha-1}\exp(-x/\beta)}{\beta^{\alpha}\Gamma(\alpha)}.
 \end{eqnarray*}
Their derivatives
 \begin{eqnarray}\label{25}
f'_M(x)&=&-\frac{\sqrt{2}x\exp(-x^2/2\sigma^2)(x^2-2\sigma^2)}{\sigma^5\sqrt{\pi}},\nonumber\\
f'_W(x)&=&-sx^{s-2}\exp(-x^s)(sx^s-s+1),\\
f'_G(x)&=&\frac{x^{\alpha-2}\exp(-x/\beta)(\beta+x-\alpha\beta)}{\beta^{\alpha+1}\Gamma(\alpha)}\nonumber
 \end{eqnarray}
are to be estimated. The Weibull and the Gamma pdfs are frequently used in a wide range of applications in engineering, signal processing, medical research, quality control, actuarial science and climatology among others. For example, most total insurance claim distributions are shaped like gamma pdfs \cite{Furman}. The gamma distribution is also used to model rainfalls \cite{Aksoy}. Gamma class pdfs, like Erlang and $\chi^2$ pdfs are widely used in modeling insurance portfolios \cite{Hurlimann}.
\par We generate Maxwell, Weibull  and Gamma i.i.d samples
with sample sizes $n\in\{100, 500, 1000 , 2000\}$ using standard Matlab generators.  To get the dependent data we generate Markov chains with the same stationary distributions using the Metropolis - Hastings algorithm \cite{Metro}. Due to the existence of the probability of rejecting a move from the previous point to the next one, the variance of such Markov sequence
$\{X_t\}$ is corrupted by the function of the latter rejecting probability (see \cite{Skold}, Theorem 3.1). The Metropolis-Hastings Markov chains \cite{Metro} are geometrically ergodic for the underlying light-tailed distributions. Hence, they satisfy the strong mixing condition \cite{Roberts}.
\par The gamma kernel estimates \eqref{2} with the optimal bandwidth \eqref{6} for the derivatives \eqref{25} can be seen in Figures~\ref{Fig1} - \ref{Fig3}. The optimal bandwidth \eqref{6} is counted for every replication of the simulation using the rule of thumb method, where as a reference density we take the gamma pdf.
\begin{figure}[Ht]
\begin{center}
\begin{minipage}[Ht]{0.49\linewidth}
\includegraphics[width=1\linewidth]{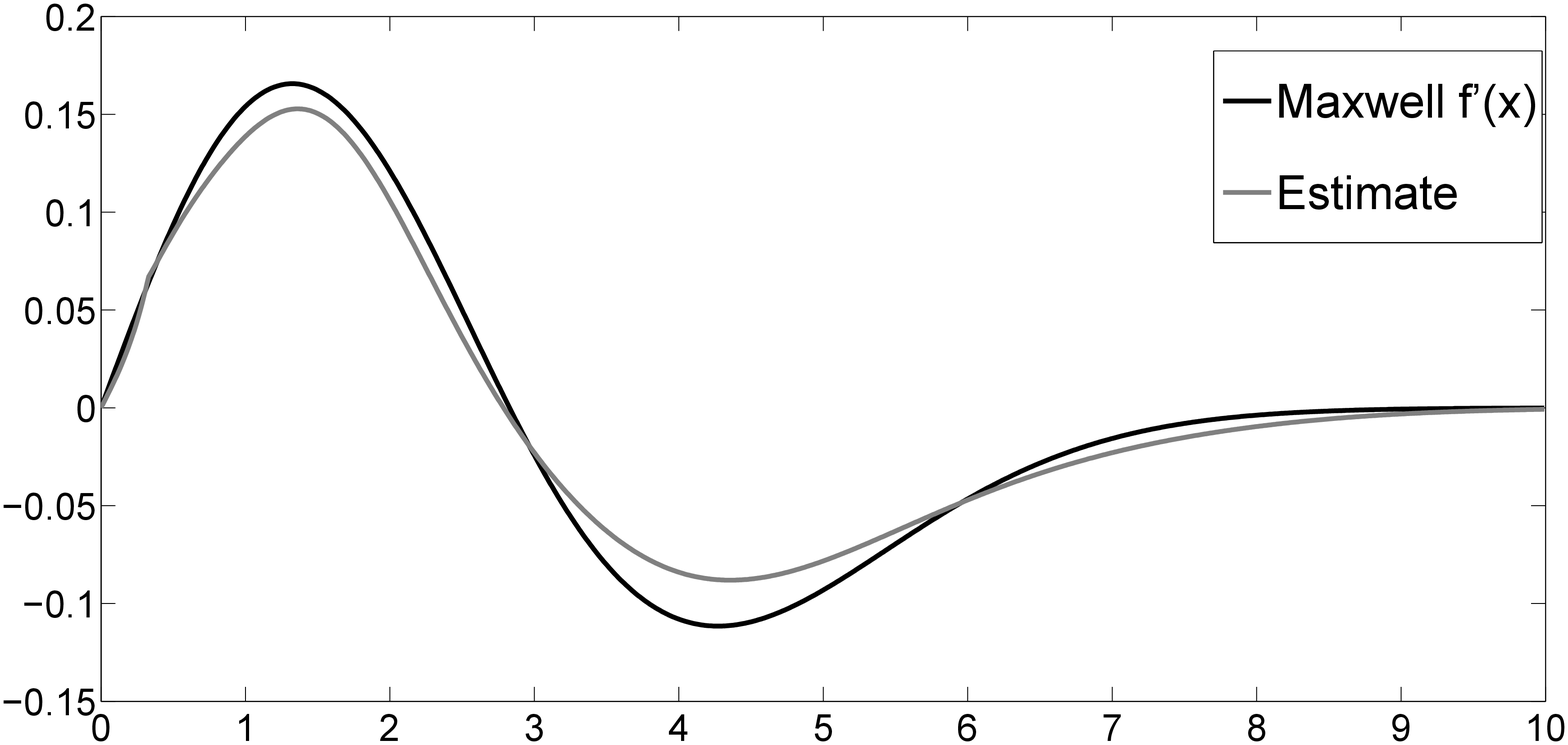}
\vspace{-4mm}
\end{minipage}
\hfill
\begin{minipage}[Ht]{0.49\linewidth}
\includegraphics[width=1\linewidth]{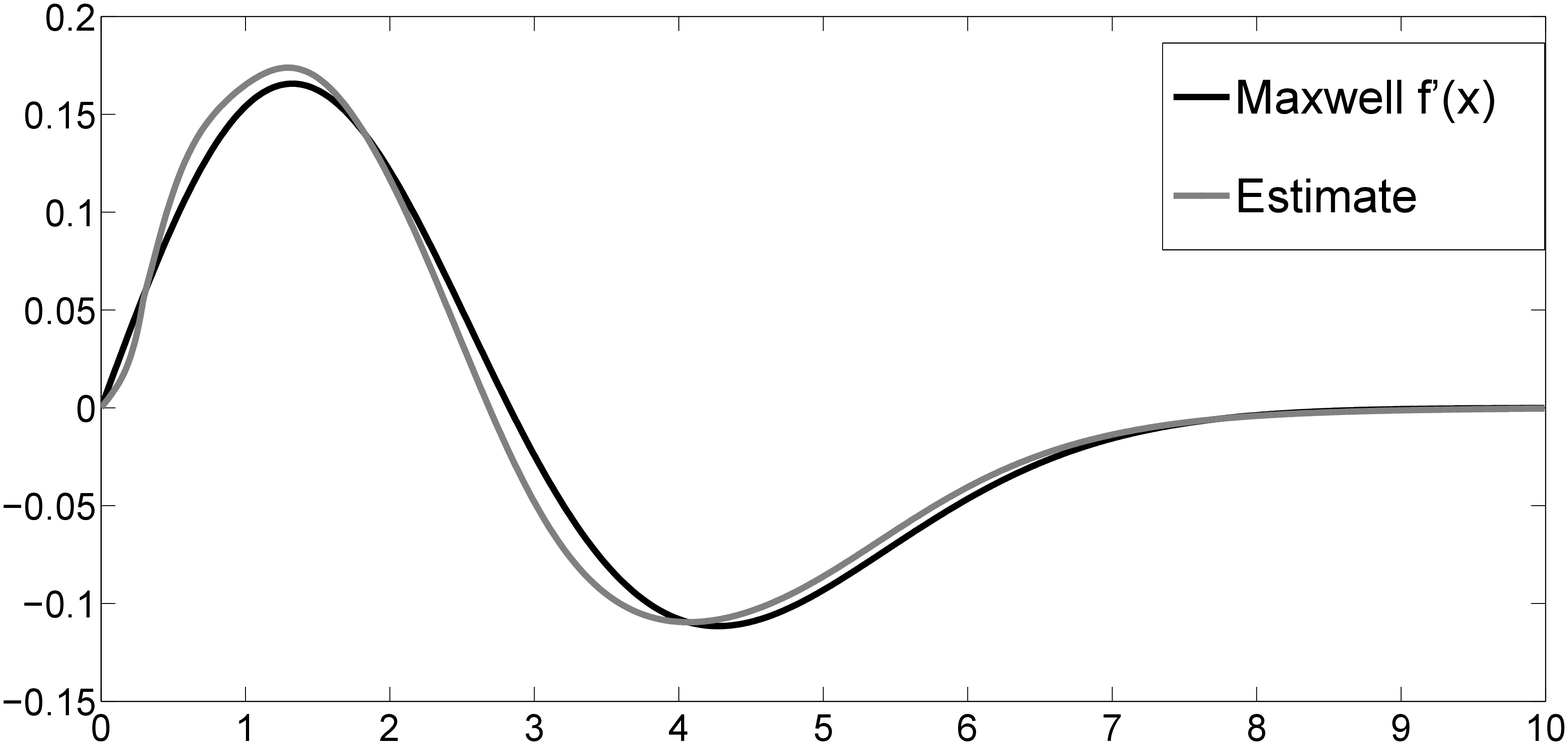}
\vspace{-4mm}
\end{minipage}
 \caption{Estimates of the Maxwell pdf derivative by i.i.d data (left) and by dependent data (right): the $f'_M(x)$ (black line),  gamma kernel estimate  from
the rule of thumb (grey  line) for the sample size $n = 2000$.\label{Fig1}}
\end{center}
\end{figure}
\begin{figure}[Ht]
\begin{center}
\begin{minipage}[Ht]{0.49\linewidth}
\includegraphics[width=1\linewidth]{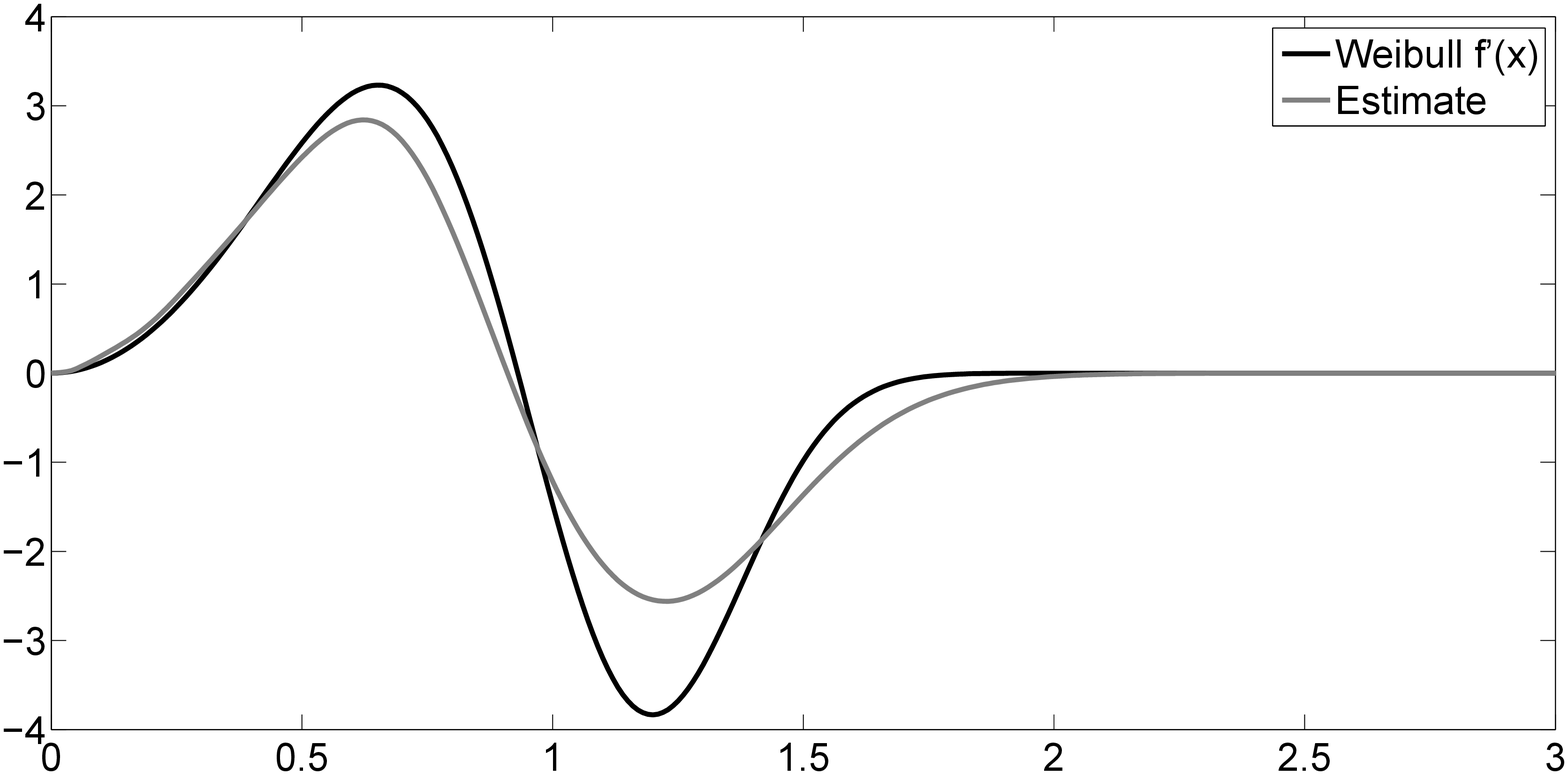}
\vspace{-4mm}
\end{minipage}
\hfill
\begin{minipage}[Ht]{0.49\linewidth}
\includegraphics[width=1\linewidth]{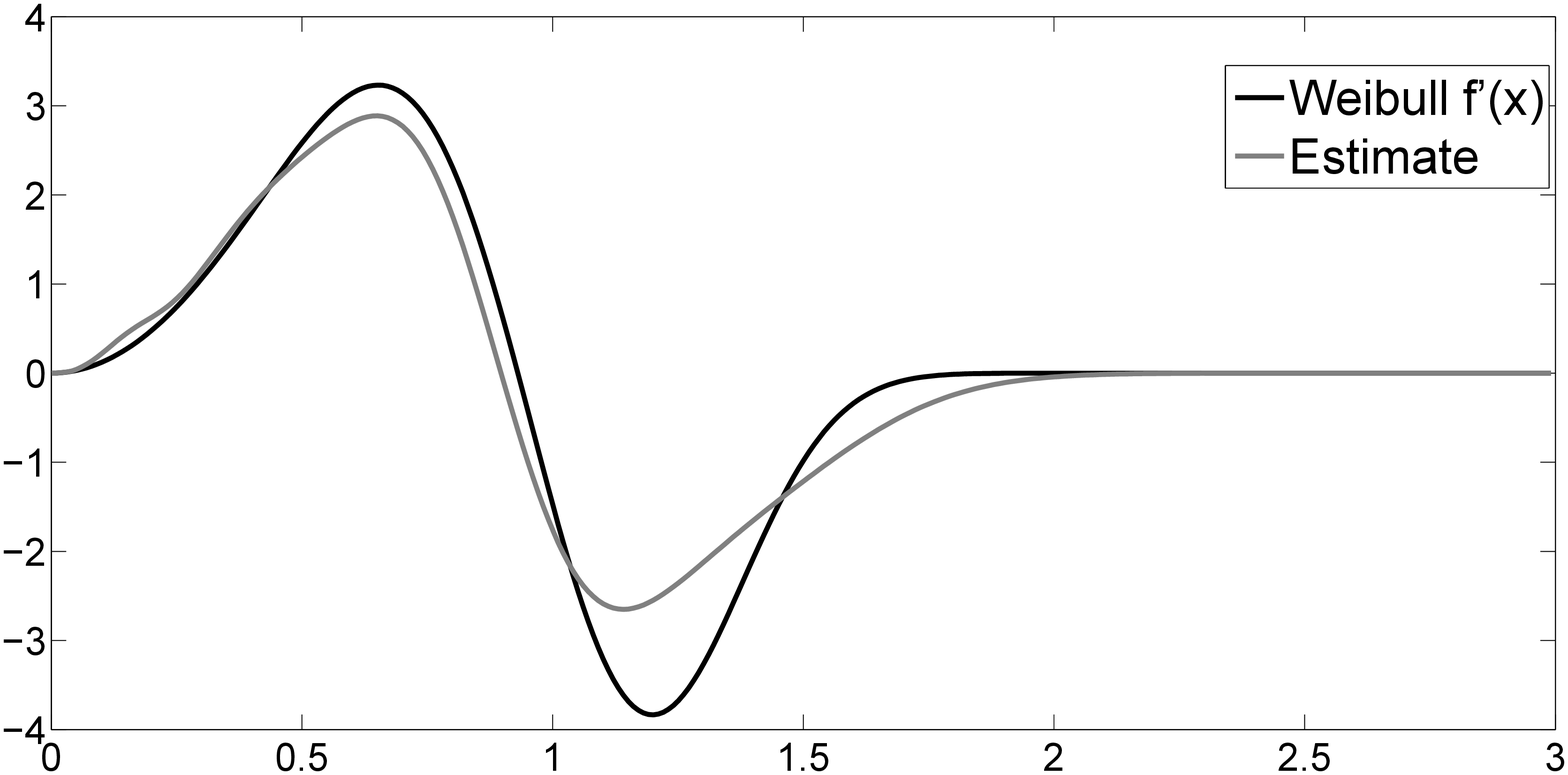}
\vspace{-4mm}
\end{minipage}
 \caption{Estimates of the Weibull pdf derivative by i.i.d data (left) and by dependent data (right): the $f'_W(x)$ (black line),  gamma kernel estimate  from
the rule of thumb (grey  line) for the sample size $n = 2000$.\label{Fig2}}
\end{center}
\end{figure}
\begin{figure}[Ht]
\begin{center}
\begin{minipage}[Ht]{0.49\linewidth}
\includegraphics[width=1\linewidth]{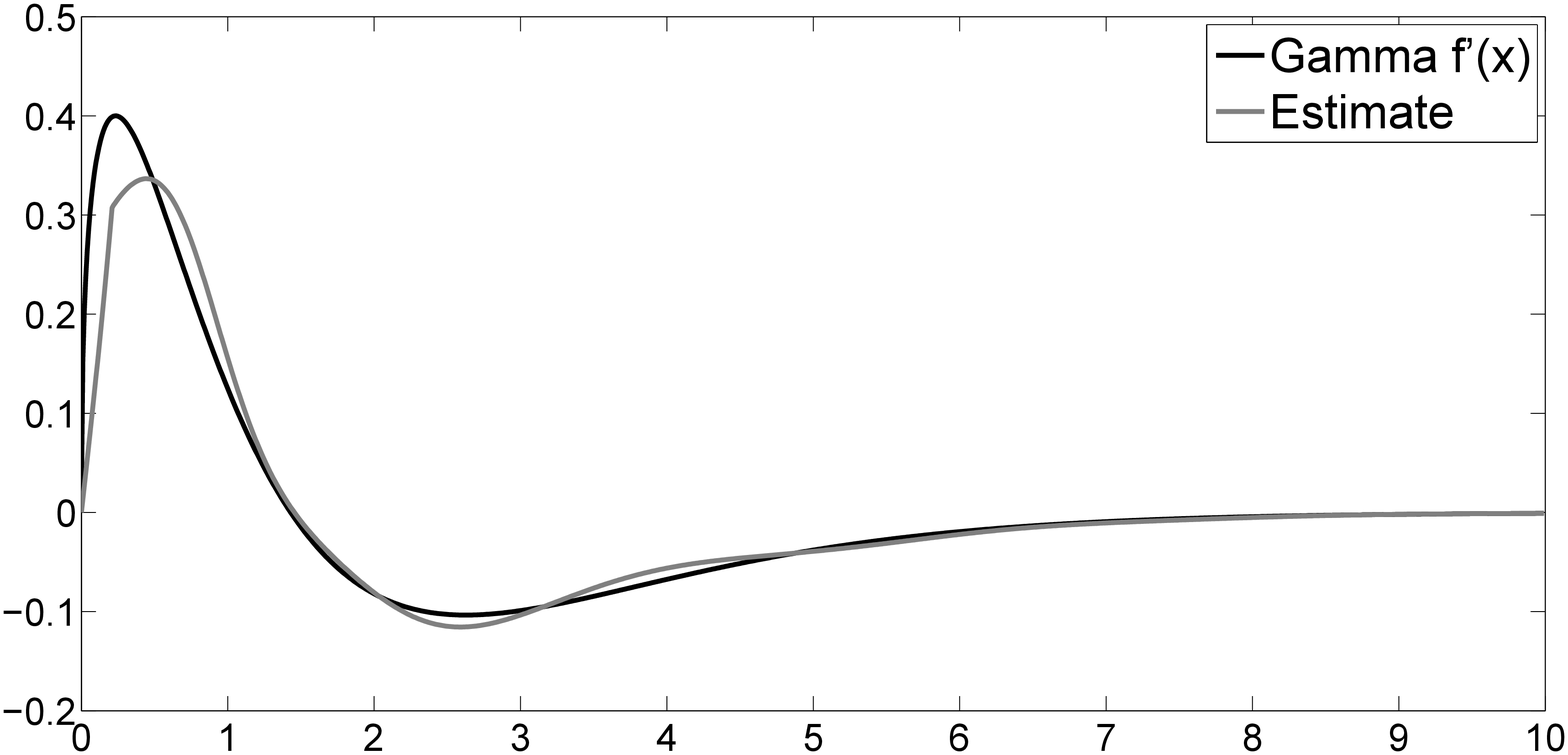}
\vspace{-4mm}
\end{minipage}
\hfill
\begin{minipage}[Ht]{0.49\linewidth}
\includegraphics[width=1\linewidth]{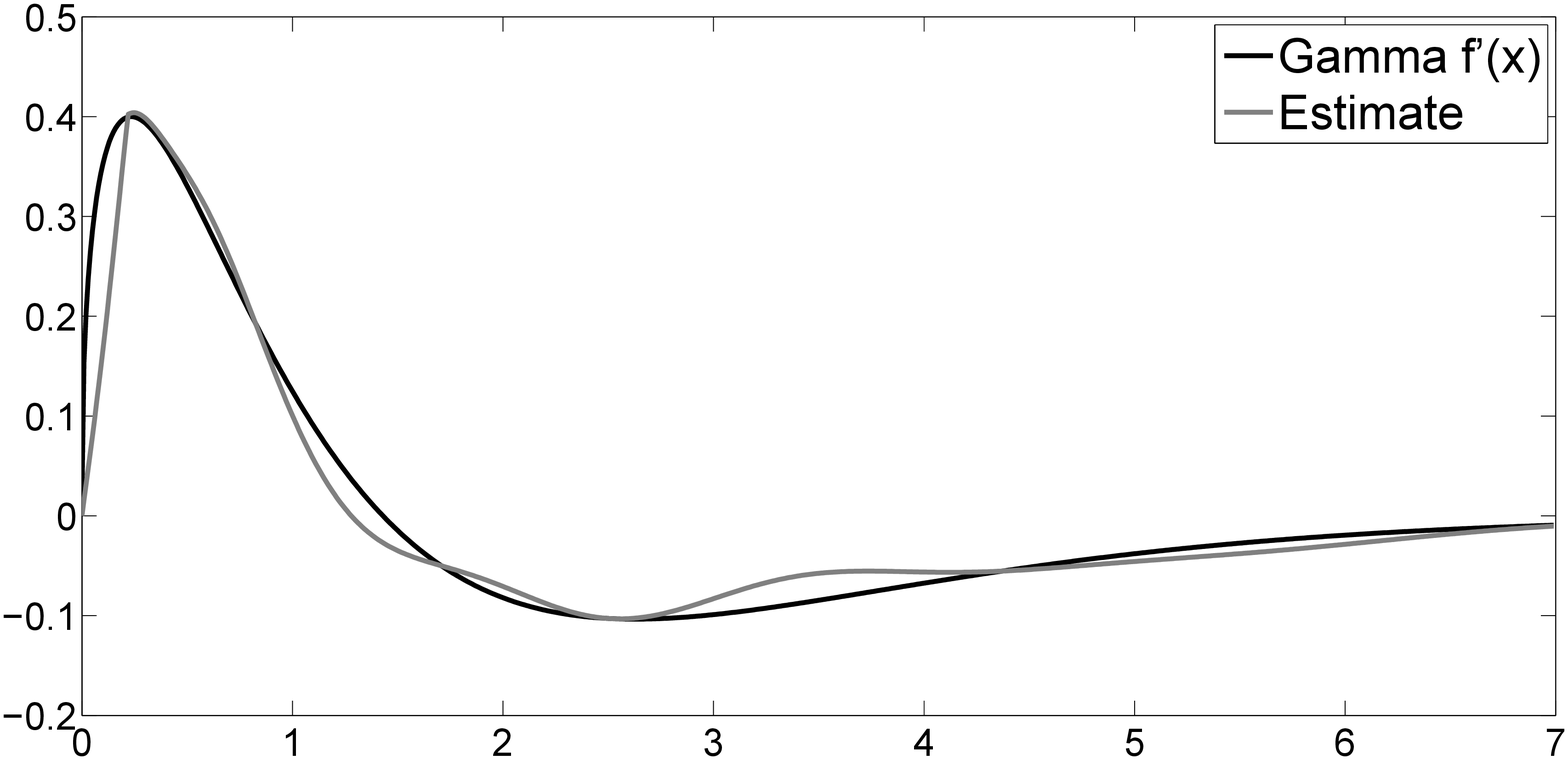}
\vspace{-4mm}
\end{minipage}
 \caption{Estimates of the Gamma pdf derivative by i.i.d data (left) and by dependent data (right): the $f'_G(x)$ (black line),  gamma kernel estimate  from
the rule of thumb (grey  line) for the sample size $n = 2000$.\label{Fig3}}
\end{center}
\end{figure}
The estimation error of the pdf derivative is calculated by the following formula
\begin{eqnarray*}
m=\int\limits_0^\infty(f'(x)-\hat{f}'(x))^2dx,
\end{eqnarray*}
where $f'(x)$ is a true derivative and $\hat{f}'(x)$ is its estimate. Values of $m'$s averaged over $500$ simulated samples and the standard deviations for the underlying distributions are given in Table~\ref{Tab1} for i.i.d r.v.s and in Table~\ref{Tab2} for dependent data.
\begin{table}[h]
\centering
\begin{tabular}{|c|c|c|c|c|}
  \hline
  n & 100 & 500 & 1000 & 2000 \\
  \hline
  Gamma & 0.032792 & 0.015208  &  0.010675  & 0.0074668  \\
   & (0.011967) & (0.0044094) & (0.0027815) & (0.0016452) \\
  \hline
  Weibull & 2.0056  & 1.1987  &  0.9157  &  0.69155   \\
  &  (0.52931) & (0.25172)& (0.18333)  &  (0.12178)   \\
  \hline
 Maxwell &   0.0077597  & 0.0035692 & 0.0028675  &  0.0020923\\
    &  (0.0033915) & (0.0015351) & (0.00099263)  & (0.00068739)\\
  \hline
\end{tabular}
\caption{Mean errors $m$ and standard deviations for i.i.d r.v.s \label{Tab1}}
\end{table}
\begin{table}[h]
\centering
\begin{tabular}{|c|c|c|c|c|}
  \hline
  n & 100 & 500 & 1000 & 2000 \\
  \hline
  Gamma &  0.039226  & 0.018124  & 0.01252   &  0.0086675    \\
  &  (0.015824) &  (0.006055) &  (0.0038485)  &  (0.0023361)   \\
  \hline
  Weibull & 2.2052 & 1.3009   & 0.97509  & 0.75382  \\
 &  (1.1585) & (0.5957)  & (0.41041) &  (0.28755)  \\
  \hline
 Maxwell & 0.0077694  & 0.0039277   &  0.002878  & 0.0027313   \\
 &  (0.006793) &  (0.0028336) &   (0.0020021)  &  (0.0016573)  \\
  \hline
\end{tabular}
\caption{Mean errors $m$ and standard deviations for strong mixed r.v.s \label{Tab2}}
\end{table}
As expected, the mean error and the standard deviation decrease when the sample size rises, and
this holds both for i.i.d and the dependent case. The performance of the
gamma kernel changes when dependence is introduced, but the results in both tables are close. The mean errors are very close due to the fact the bandwidth parameter is selected to minimize this error. However, the standard deviations for the dependent data are higher than for the i.i.d r.v.s. For example, for the sample size of $500$ the mean errors and the standard deviations for the Maxwell pdf for the i.i.d r.v.s are $0.0035692$ $(0.0015351)$ and for dependent r.v.s $0.0039277$ $(0.0028336)$.
They differ due to the contribution of the Metropolis-Hastings rejecting probability. This difference is less pronounced for larger sample sizes.
\par The Metropolis-Hastings algorithm gives opportunity to generate AR processes  with known pdfs. As a consequence we know their derivatives and can find mean errors and
standard deviations of the gamma-kernel density derivatives estimates for the dependent data. In the case when we consider the noise distribution  $\{\epsilon\}$
of the AR model \eqref{15} and the autoregressive parameter $\rho$ that influences on the dependence rate \eqref{alpha}, we cannot indicate in general the true pdf of the process.
Hence, we consider the histogram based on $200000$ observations as a true pdf.
As the noise distribution $\{\epsilon\}$ let us take the Gamma distribution ($\alpha=1.5,\beta=1$) and the Maxwell distribution ($\sigma=1$). In
\cite{TaufikBouezmarnia:Rom2} it was proved that, as in the i.i.d case, the gamma-kernel estimator of the pdf achieves the same optimal rate of convergence
in terms of the mean integrated squared error as for strongly mixed r.v.s.
For the various parameters $\rho\in\{0.1, 0.2, 0.3, 0.4\}$ the gamma estimates for the densities of the AR models are given in Figures~\ref{Fig4}-\ref{Fig5}.
\begin{figure}[Ht]
\begin{center}
\begin{minipage}[Ht]{0.49\linewidth}
\includegraphics[width=1\linewidth]{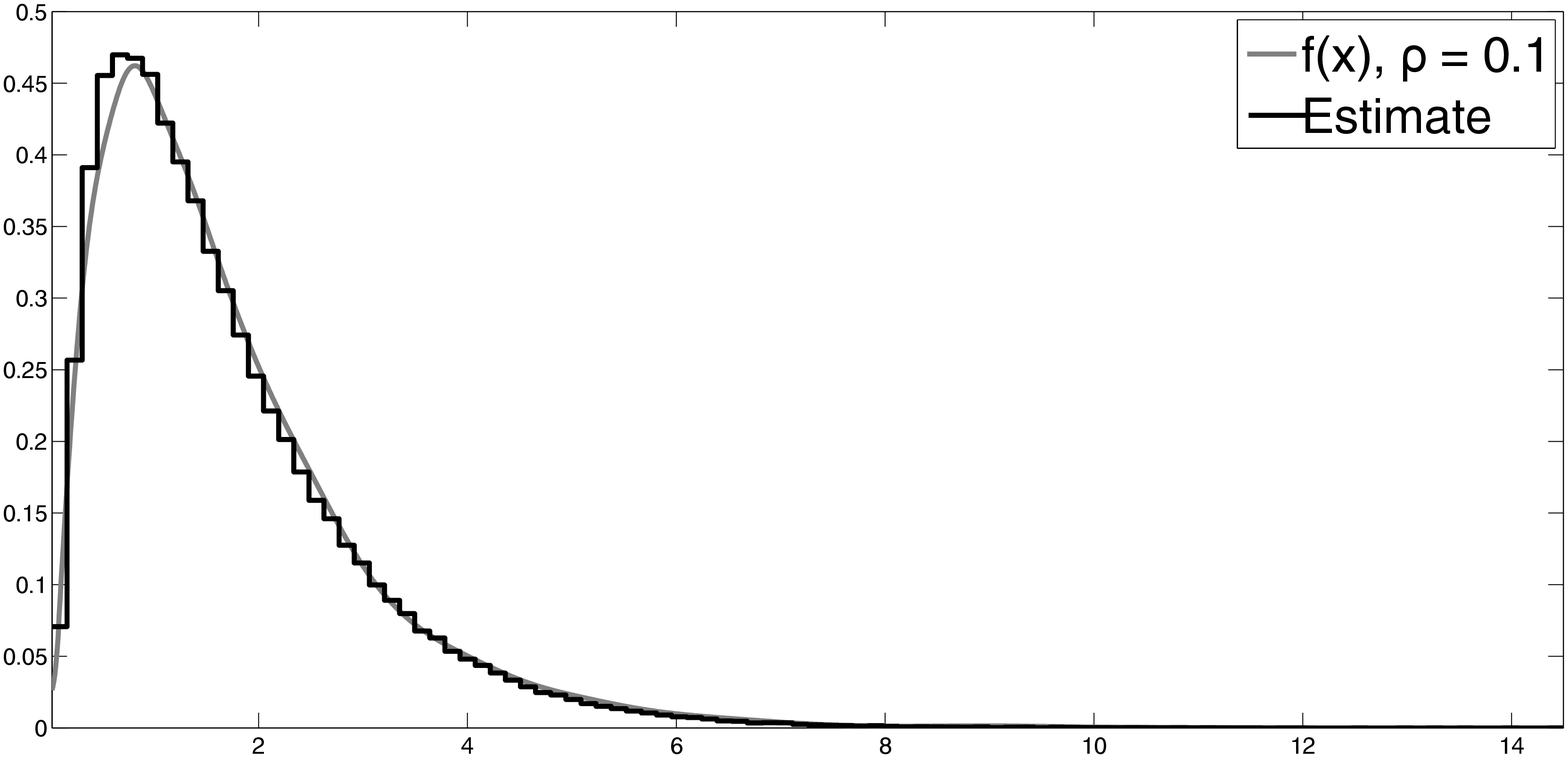}
\vspace{-4mm}
\end{minipage}
\hfill
\begin{minipage}[Ht]{0.49\linewidth}
\includegraphics[width=1\linewidth]{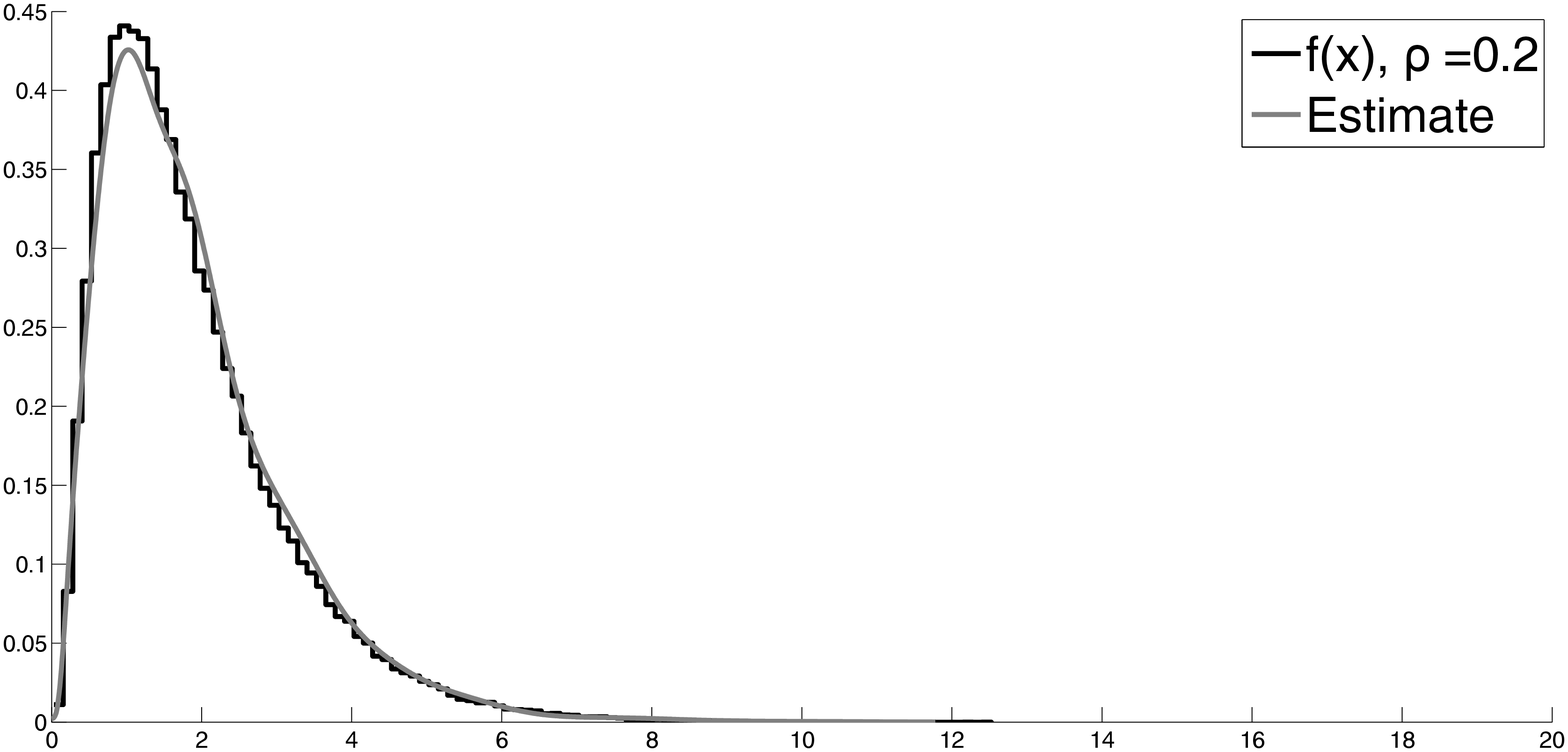}
\vspace{-4mm}
\end{minipage}
\vfill
\begin{minipage}[Ht]{0.49\linewidth}
\includegraphics[width=1\linewidth]{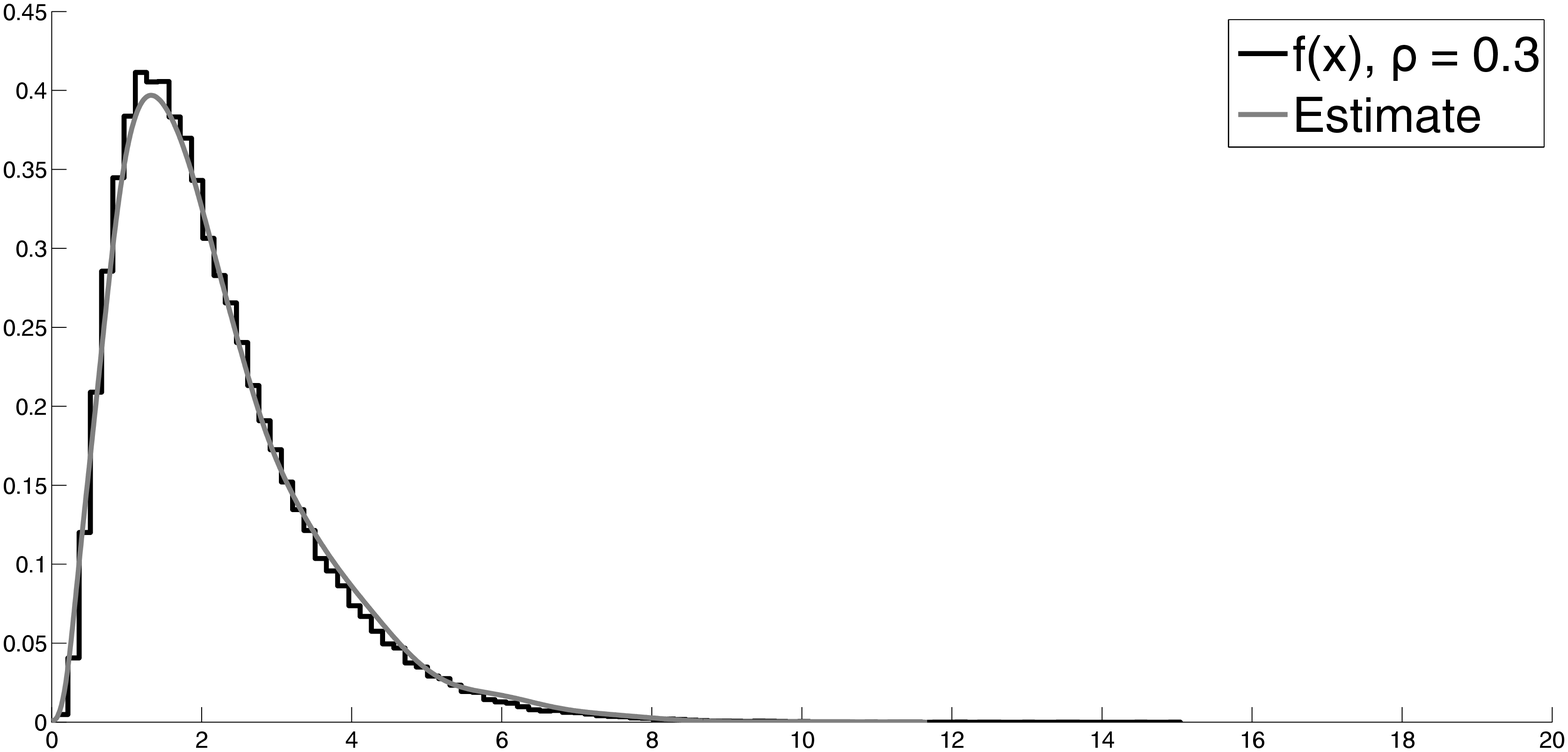}
\vspace{-4mm}
\end{minipage}
\hfill
\begin{minipage}[Ht]{0.49\linewidth}
\includegraphics[width=1\linewidth]{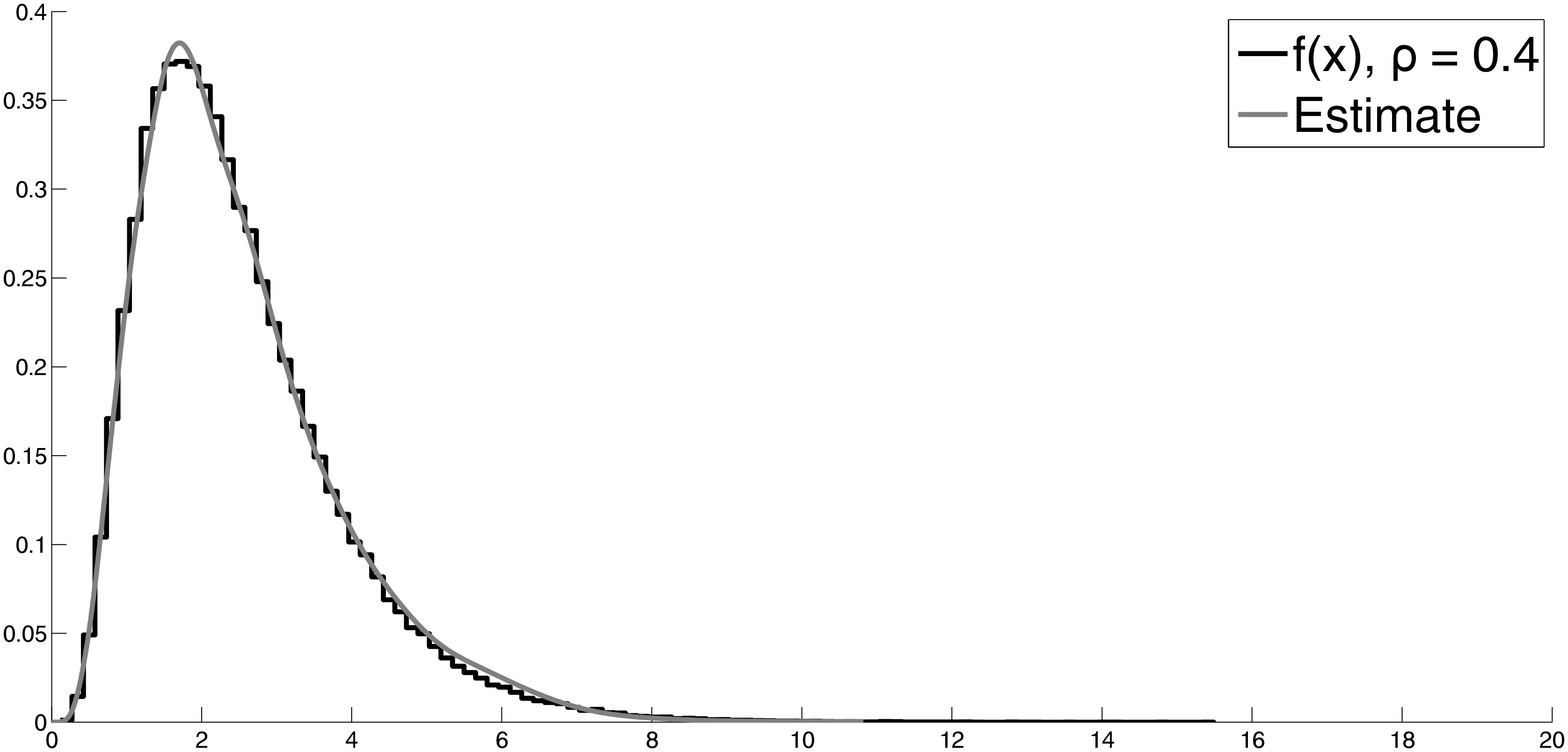}
\vspace{-4mm}
\end{minipage}
 \caption{Gamma-kernel estimates of the pdf of the AR model with the Gamma noise and  $\rho\in\{0.1, 0.2, 0.3, 0.4\}$ for the sample size $n = 2000$.\label{Fig4}}
\end{center}
\end{figure}
\begin{figure}[Ht]
\begin{center}
\begin{minipage}[Ht]{0.49\linewidth}
\includegraphics[width=1\linewidth]{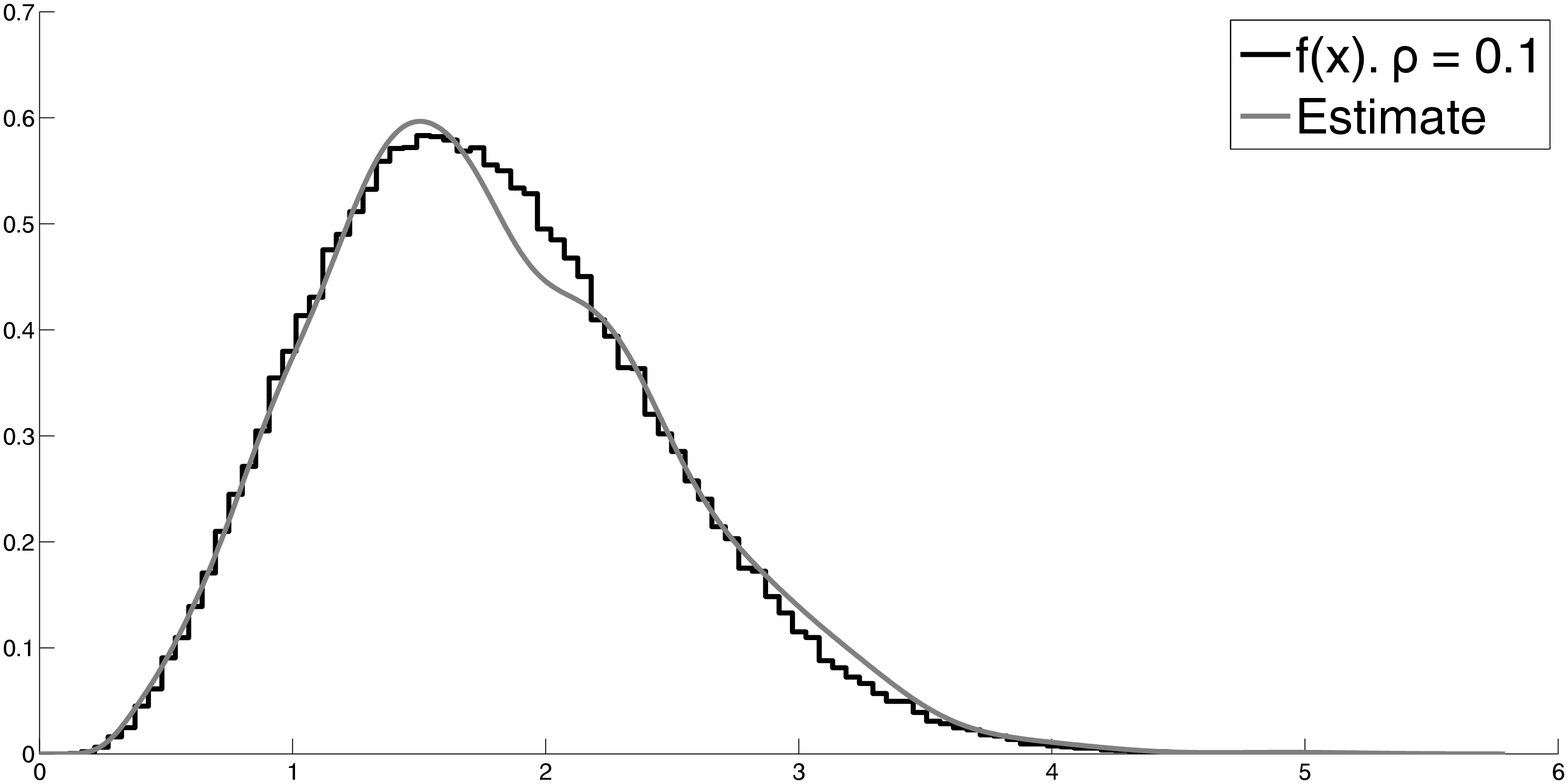}
\vspace{-4mm}
\end{minipage}
\hfill
\begin{minipage}[Ht]{0.49\linewidth}
\includegraphics[width=1\linewidth]{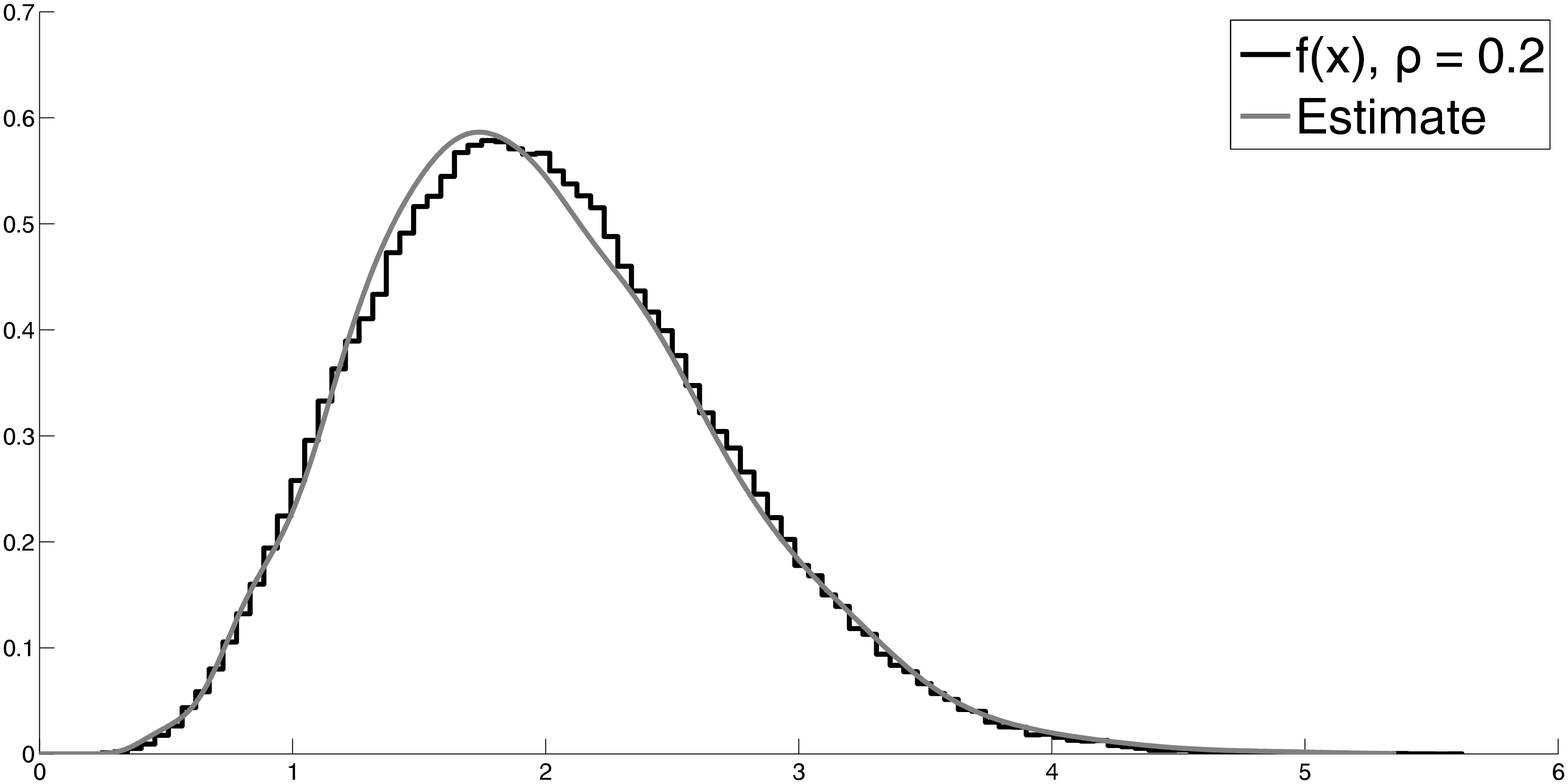}
\vspace{-4mm}
\end{minipage}
\vfill
\begin{minipage}[Ht]{0.49\linewidth}
\includegraphics[width=1\linewidth]{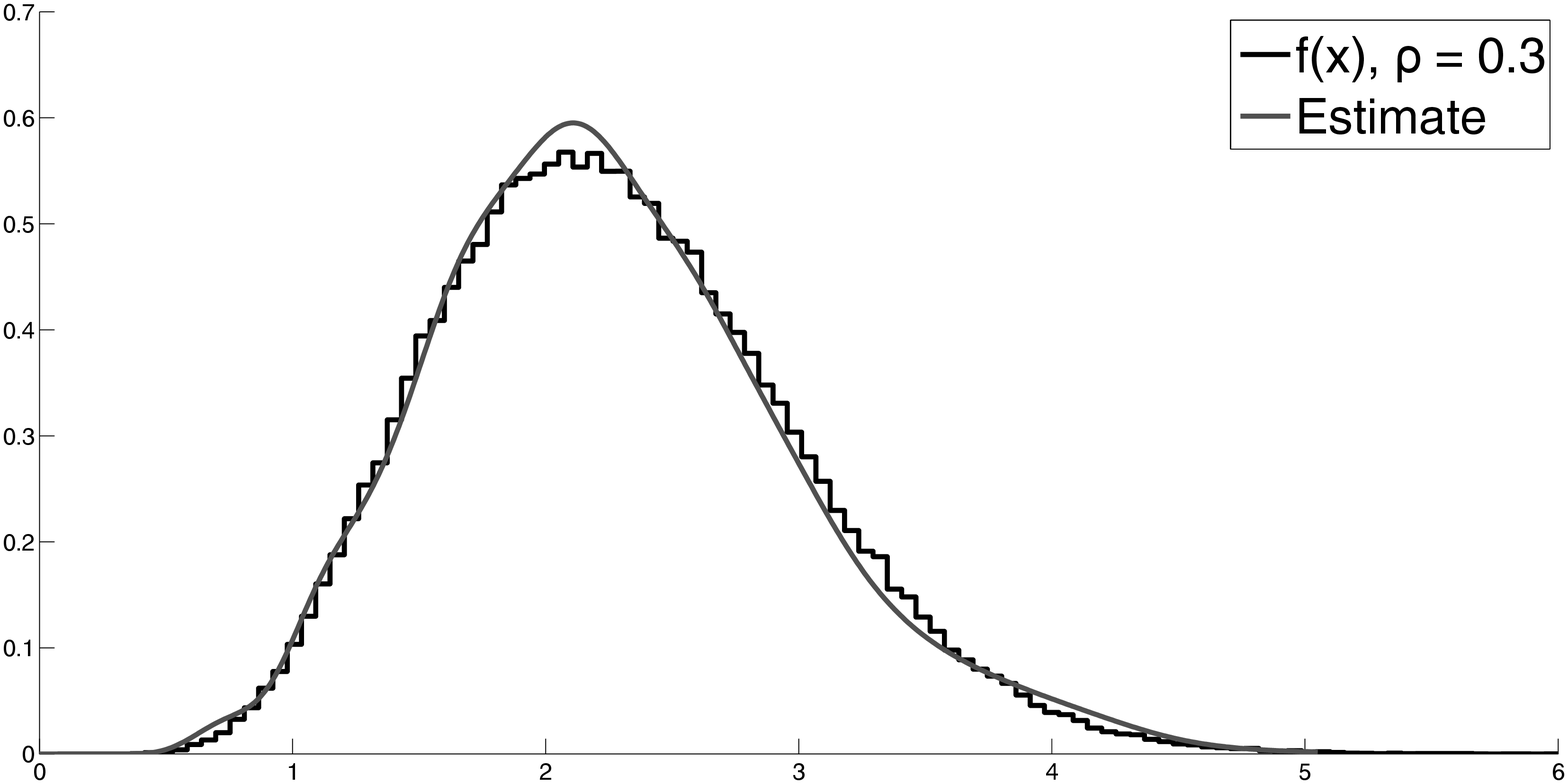}
\vspace{-4mm}
\end{minipage}
\hfill
\begin{minipage}[Ht]{0.49\linewidth}
\includegraphics[width=1\linewidth]{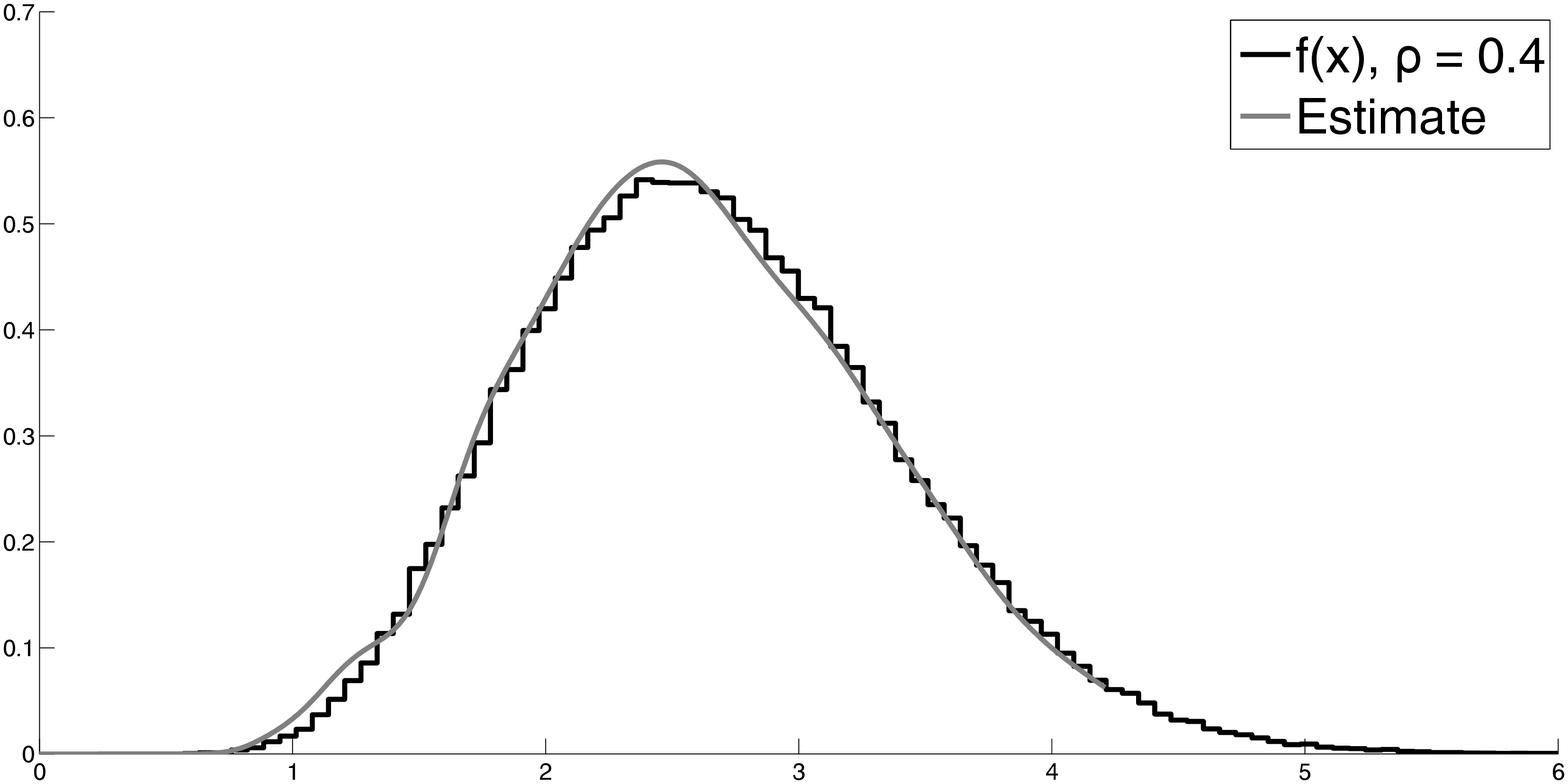}
\vspace{-4mm}
\end{minipage}
 \caption{Gamma-kernel estimates of the pdf of the AR model with the Maxwell noise and $\rho\in\{0.1, 0.2, 0.3, 0.4\}$ for the sample size $n = 2000$.\label{Fig5}}
\end{center}
\end{figure}
Since the gamma-kernel estimators perform good for the various dependence rates it is also true for the gamma-kernel pdf derivative estimators, but the bandwidth parameter must be selected differently.
\par Hence, this findings confirms the fact that the covariance term \eqref{4} of the pdf derivative
 is negligible in comparison with its variance and implies that one can use the same optimal bandwidth \eqref{6}, both for independent and strongly mixed dependent data.

\begin{acknowledgments}
I am grateful to my supervisor DrSci Alexander Dobrovidov for an interesting topic. The work
was partly supported by the Russian Foundation for Basic Research, grant 13-08-00744 A.

\end{acknowledgments}
\section{APPENDIX}
\begin{proof}[Proof of Lemma~\ref{Lem1}]
Taking an integral from \eqref{900} we get
\begin{eqnarray}\label{MISE}MISE(\widehat{f'}(x))= \int\limits_{0}^{\infty}(B(x)^2+V(x)+C(x))dx,
\end{eqnarray}
where
\begin{eqnarray}\label{C(x)}
C(x)=\frac{2}{n}\sum\limits_{i=1}^{n-1}\left(1-\frac{i}{n}\right)\mbox{cov}(K'_b(X_1),K'_b(X_{1+i})).
\end{eqnarray}
To evaluate  the covariance we shall apply Davydov's inequality
\begin{eqnarray}\label{Dovid}|\mbox{cov}(K'_b(X_1),K'_b(X_{1+i}))|\leq 2\pi \alpha(i)^{1/r}\parallel K'_b(X_1) \parallel_q \parallel K'_b(X_{1+i}) \parallel_p,
\end{eqnarray}
where $p^{-1}+q^{-1}+r^{-1}=1$, $1\leq p, q, r\leq\infty$, \cite{Bosq:96}.
\par The latter norm for the case $x\geq 2b$ is determined by
\begin{eqnarray}\label{K1par}\parallel K'_b(X_1) \parallel_q&=&\left(\int\left(\frac{1}{b}K(y)L_1(y)\right)^q f(y)dy\right)^{1/q}\\\nonumber
&=&\frac{1}{b}\left(\mathsf{E}\left(K(\xi_1)^{q-1}L_1(\xi_1)^q f(\xi_1)\right)\right)^{1/q},
\end{eqnarray}
where $L_1(t)$ is introduced in \eqref{L}. The kernel $K(\xi_1)$ was used in \eqref{K1par} as a density function and $\xi_1$ is a $Gamma(\rho_1(x),b)$ random variable.
\par In the case $ x\in [0,2b)$, similarly we have
 \begin{eqnarray}\label{K2par}\parallel K'_b(X_1) \parallel_q&=&\left(\int\left(\frac{x}{2b^2}K(y)L_2(y)\right)^q f(y)dy\right)^{1/q}\\\nonumber
 &=&\frac{x}{2b^2}\left(\mathsf{E}\left(K(\xi_2)^{q-1}L_2(\xi_2)^q f(\xi_2)\right)\right)^{1/q},
\end{eqnarray}
where $L_2(t)$ is determined by \eqref{L}, and $\xi_2$ is a $Gamma(\rho_2(x),b)$ random variable.
Expressions \eqref{K1par} and \eqref{K2par} are constructed similarly, thus to a certain point, we will not make differences between them.
\par By the standard theory of the gamma distribution it is known that $\mu = \mathsf{E}(\xi)= \rho_{b}(x) b$ and the
variance is given by $var(\xi)=\rho_{b}(x) b^2$. For simplicity, we further use the notation $\rho$ instead of $\rho_{b}(x)$ defined in \eqref{rho}.
\par The Taylor expansion of both mathematical expectations in \eqref{K1par}, \eqref{K2par} in the neighborhood  of $\mu$ is represented by
\begin{eqnarray*}\mathsf{E}\left(K(\xi)^{q-1}L(\xi)^q f(\xi)\right)
&=&K(\mu)^{q-1} L(\mu)^q f(\mu)+(K(\xi)^{q-1} L(\xi)^q f(\xi))'|_{\xi=\mu}\mathsf{E}(\xi-\mu)\\
&+&\left(K(\xi)^{q-1} L(\xi)^q f(\xi)\right)''|_{\xi=\mu}\frac{\mathsf{E}(\xi-\mu)^2}{2}+o\left(\mathsf{E}(\xi-\mu)^2\right).
\end{eqnarray*}
In the case when $x\geq2b$, $\mu=\rho b=x$, $var(\xi)=\rho b^2=xb$, we get
\begin{eqnarray*}\label{24}&&\mathsf{E}\left(K(\xi)^{q-1}L(\xi)^q f(\xi)\right)=
 \frac{K(x)^{q-1} }{b}\Bigg(qL(x)^{q+1}f'(x)-L(x)^qf(x)L'(x)\\
 &-&L(x)^{q+1}f'(x)+
bL(x)^qf''(x)+q^2L(x)^qf(x)L'(x)+bq^2L(x)^{q-2}(L'(x))^2f(x)\\
&+&2bqL(x)^{q-1}L'(x)f'(x)+bqL(x)^{q-1}f(x)L''(x)-bqL(x)^{q-2}(L'(x))^2\Bigg)\\
&+&\frac{K(x)^{q-1}L(x)(q-1)}{b^2}\Bigg((q-1)f(x)L(x)^{q+1}+bL(x)^qf'(x)\\
&+&bqL(x)^{q-1}f(x)L'(x)\Bigg)+o\left(b^2\right).
\end{eqnarray*}
Using Stirling's formula \begin{eqnarray*}\label{stirling}\Gamma(z)=\sqrt{\frac{2\pi}{z}}\left(\frac{z}{e}\right)^z\left(1+O\left(\frac{1}{z}\right)\right),\end{eqnarray*} we can  rewrite the kernel function as
\begin{eqnarray*}&&K(t)=\frac{t^{\rho-1}\exp(-t/b)}{b^{\rho}\Gamma(\rho)}=\frac{t^{\rho-1}\exp(-t/b)\exp(\rho)}{b^{\rho}\sqrt{2\pi}\rho^{\rho-\frac{1}{2}}(1+O(1/\rho))}.
\end{eqnarray*}
 Taking $\rho=\rho_1(x)$ according to \eqref{rho}, $t=x$, it holds
\begin{eqnarray*}&&K(\rho_1(x) b)=\frac{1}{\sqrt{2\pi}}\frac{x^{x/b-1}\exp((x-x)/b)}{b^{\frac{x}{b}}\frac{x}{b}^{\frac{x}{b}-\frac{1}{2}}(1+O(b/x))}=\frac{x^{-\frac{1}{2}}b^{-\frac{1}{2}}}{\sqrt{2\pi}(1+O(b/x))}.
\end{eqnarray*}
Hence, its upper bound is given by
\begin{eqnarray}\label{K1}&&K(x)\leq\frac{1}{\sqrt{2\pi x b}}.
\end{eqnarray}
Next, using the property of the Digamma function $\Psi(x) = \ln(x) - \frac{1}{2x} - \frac{1}{12x^2} + \frac{1}{120x^4}+O(1/x^6)$,
the first equation in \eqref{L} can de rewritten as
\begin{eqnarray}\label{L1}L_1(\rho_1 b)&=&\ln(\rho_1 b)-\ln(b)-\Psi(\rho_1)=
\frac{b}{2x}+\frac{b^2}{12x^2}+o(b^2).
\end{eqnarray}
Then substituting \eqref{24} in \eqref{K1par}
and using the expressions \eqref{K1} and \eqref{L1}, we deduce
\begin{eqnarray*}&&\parallel K'_b(X_1) \parallel_q\leq\pi^{\frac{1-q}{2q}}(2x)^{\frac{1-q}{2q}-1}b^{\frac{1-q}{2q}}
\Bigg(b^2C_2(q,x)+bC_1(q,x)+C_3(q,x)\Bigg)^{1/q}\!\!+\!o(b^2),
\end{eqnarray*}
where we used the notations
\begin{eqnarray}\label{C123}C_1(q,x)&=&-f(x)\frac{2q^3-9q^2+4q-33}{24x}-f'(x)\frac{q+1}{2}+f''(x)\frac{x}{2},\\\nonumber
C_2(q,x)&=&f(x)\frac{2q+54x-q^2x+21q^3x+q^4x+93qx}{144x^3}\\\nonumber
&-&f'(x)\frac{(q+1)^2}{12x}+f''(x)\frac{q+1}{12},\\\nonumber
C_3(q,x)&=&-f(x)\frac{(q+1)(q-2)}{2}.
\end{eqnarray}
The same steps can be done for $\parallel K'_b(X_{1+i}) \parallel_p$ from \eqref{Dovid}.
Then, if $p=q$ holds, one can represent Davydov's inequality \eqref{Dovid}  as
\begin{eqnarray}\label{cov}&&|cov(K'_b(X_1),K'_b(X_{1+i}))|\leq \\\nonumber
&\leq&2\pi \alpha(i)^\frac{1}{r}\pi^{\frac{1-q}{q}}(2x)^{\frac{1-q}{q}-2}b^{\frac{1-q}{q}}
\Bigg(b^2C_2(q,x)+bC_1(q,x)+C_3(q,x)\Bigg)^{2/q}+o(b^2).
\end{eqnarray}
Using \eqref{cov} and taking $p=q=2+\delta$, $r=\frac{2+\delta}{\delta}$ it can be deduced that the covariance \eqref{C(x)} is given by
\begin{eqnarray*}&&|C(x)|=\left|\frac{2}{n}\sum\limits_{i=1}^{n-1}\left(1-\frac{i}{n}\right)cov(K'_b(X_1),K'_b(X_{1+i}))\right|\\
&\leq&
\Bigg|\Bigg(2^{-\frac{2\delta+3}{\delta+2}}\pi^{\frac{1}{\delta+2}}x^{-\frac{3\delta+5}{\delta+2}}\frac{b^{-\frac{\delta+1}{\delta+2}}}{n}
\Bigg(b^2C_2(\delta,x)+bC_1(\delta,x)+C_3(\delta,x)\Bigg)^{\frac{2}{2+\delta}}\Bigg)\\
&\cdot&\sum\limits_{i=1}^{n-1}\left(1-\frac{i}{n}\right)\alpha(i)^{\frac{\delta}{2+\delta}}\Bigg|+o(b^2).
\end{eqnarray*}
Then we can estimate the covariance by the previous expressions
\begin{eqnarray*}\label{covsum}\nonumber|C(x)|&\leq&S(b,x,\delta,n)\sum\limits_{\tau=2}^{n}\left(1-\frac{\tau-1}{n}\right)\alpha(\tau-1)^{\frac{\delta}{2+\delta}}+o(b^2)\\\nonumber
&\leq&S(b,x,\delta,n)\sum\limits_{\tau=2}^{\infty}\alpha(\tau-1)^{\frac{\delta}{2+\delta}}+o(b^2)
\leq S(b,x,\delta,n)\int\limits_{1}^{\infty}\alpha(\tau)^{\frac{\delta}{2+\delta}}d\tau+o(b^2),
\end{eqnarray*}
where we used the following notation
\begin{eqnarray*}S(b,x,\delta,n)&=&2^{-\frac{2\delta+3}{\delta+2}}\pi^{\frac{1}{\delta+2}}x^{-\frac{3\delta+5}{\delta+2}}\frac{b^{-\frac{\delta+1}{\delta+2}}}{n}
\Bigg(b^2C_2(\delta,x)+bC_1(\delta,x)+C_3(\delta,x)\Bigg)^{\frac{2}{2+\delta}}.
\end{eqnarray*}
Let us denote $\frac{\delta}{2+\delta}=\upsilon$, $0<\upsilon<1$. Then, in this notations, we get
the estimate of the covariance
\begin{eqnarray*}&&|C(x)|\leq\\
&\leq&\Bigg(2^{-\frac{\upsilon+3}{2}}\pi^{\frac{1-\upsilon}{2}}x^{-\frac{\upsilon+5}{2}}\frac{b^{-\frac{\upsilon+1}{2}}}{n}
\Bigg(bC_1(\upsilon,x)+C_3(\upsilon,x)\Bigg)^{1-\upsilon}+o(b^2)\Bigg)\int\limits_{1}^{\infty}\alpha(\tau)^{\upsilon}d\tau.
\end{eqnarray*}
By  $0<\upsilon<1$ then it follows
\begin{eqnarray*}&&|C(x)|\sim\frac{1}{n}b^{-\frac{\upsilon+1}{2}}.
\end{eqnarray*}
\begin{remark}The main contribution to MISE \eqref{MISE} is provided by the part corresponding to $x\geq2b$, so  we will not do similar calculations here and further for $x\in[0,2b)$ as $b\rightarrow0$.
\end{remark}
\end{proof}
\begin{proof}[Proof of Theorem~\ref{thm}]\label{ap2}
Regarding the dependent case it is known that the MISE contains the bias, the variance and the covariance.
By \eqref{5} it follows that the integrated sum of the squared bias and variance is the following expression
\begin{eqnarray}\label{13}\nonumber&&\int\limits_0^\infty(B(x)^2+V(x))dx=\frac{b^2}{16}\int\limits_{0}^\infty P(x) dx
\\
&+&  \int\limits_{0}^{\infty} \frac{n^{-1}b^{-\frac{3}{2}}x^{-\frac{3}{2}}}{4\sqrt{\pi}}\left(f(x)+\frac{b}{2}\left(\frac{f(x)}{x}-f'(x)\right)\right)dx+o(b^2 + n^{-1}b^{-\frac{3}{2}}).
\end{eqnarray}
This corresponds to the independent case.
\par By integration of \eqref{4} we get the upper bound of the integrated covariance
\begin{eqnarray}\label{14}&&\int\limits_0^\infty \! C(x)dx\leq\!\!\int\limits_{0}^{\infty}\!\!\!\Bigg(2^{-\frac{\upsilon+3}{2}}\pi^{\frac{1-\upsilon}{2}}x^{-\frac{\upsilon+5}{2}}\frac{b^{-\frac{\upsilon+1}{2}}}{n}
C_3(\upsilon,x)^{1-\upsilon}+o(b^2)\Bigg)\!\!\!\int\limits_{1}^{\infty}\alpha(\tau)^{\upsilon}d\tau dx.
\end{eqnarray}
Combining \eqref{13} and \eqref{14}, one  can write
\begin{eqnarray*} &&MISE(f'(x))\leq\int\limits_{0}^{\infty} \frac{n^{-1}b^{-3/2}x^{-3/2}}{4\sqrt{\pi}}\left(f(x)+\frac{b}{2}\left(\frac{f(x)}{x}-f'(x)\right)\right)dx\\ \nonumber
&+& \int\limits_{0}^{\infty}2^{-\frac{\upsilon+3}{2}}\pi^{\frac{1-\upsilon}{2}}x^{-\frac{\upsilon+5}{2}}\frac{b^{-\frac{\upsilon+1}{2}}}{n}
C_3(\upsilon,x)^{1-\upsilon}dx\int\limits_{1}^{\infty}\alpha(\tau)^{\upsilon}d\tau \\\nonumber
&+&\frac{b^2}{16}\int\limits_{0}^{\infty} P(x)dx+o(b^2 + n^{-1}b^{-\frac{5}{2}}).
\end{eqnarray*}
The derivative of this expression in b leads to
\begin{eqnarray}&&\label{100500} \nonumber\frac{b}{8}\int\limits_{0}^{\infty} P(x)dx-
\frac{3n^{-1}b^{-\frac{5}{2}}}{8\sqrt{\pi}}\int\limits_{0}^{\infty} x^{-\frac{3}{2}}f(x)dx\\
&+&\frac{n^{-1}b^{-\frac{3}{2}}}{16\sqrt{\pi}}\int\limits_{0}^{\infty} x^{-\frac{3}{2}}\left(\frac{f(x)}{x}-f'(x)\right)dx\\\nonumber
&-& \int\limits_{0}^{\infty}\frac{\upsilon+1}{2}2^{-\frac{\upsilon+3}{2}}\pi^{\frac{1-\upsilon}{2}}x^{-\frac{\upsilon+5}{2}}\frac{b^{-\frac{\upsilon+3}{2}}}{n}
C_3(\upsilon,x)^{1-\upsilon}dx\int\limits_{1}^{\infty}\alpha(\tau)^{\upsilon}d\tau =0.\nonumber
\end{eqnarray}
Since $0<\upsilon<1$ holds as in Lemma \ref{Lem1},
the third  term in \eqref{100500} by $b$ has the worst rate
\begin{eqnarray*}c_1b^{-\frac{\upsilon+3}{2}}&=&O\left(b^{-\frac{3}{2}}\right),
\end{eqnarray*}
where $c_1$ is a constant.
\par Neglecting terms with $b^{-3/2}$ and  $b^{-\frac{\upsilon+3}{2}}$  in comparison to the term containing $b^{-5/2}$, we simplify the equation
\begin{eqnarray*}
&& \frac{b^{7/2}}{8}\int\limits_{0}^{\infty} P(x)dx-
\frac{3n^{-1}}{8\sqrt{\pi}}\int\limits_{0}^{\infty} x^{-\frac{3}{2}}f(x)dx+o(b^{7/2})=0.
\end{eqnarray*}
The optimal $b=o(n^{-2/7})$ is the same as in \eqref{6}. Let us insert such $b$ in \eqref{10}
\begin{eqnarray}\label{12_1} &&MISE_{opt}(\hat f'(x))=\int_0^\infty\frac{P(x)n^{-\frac{4}{7}}}{16}
T^{\frac{4}{7}}dx+\int\limits_{0}^{\infty} \frac{n^{-4/7}T^{-3/7}x^{-3/2}}{4\sqrt{\pi}}f(x)dx\\\nonumber
&+&
\int\limits_{0}^{\infty} \frac{n^{-6/7}T^{-1/7}x^{-3/2}}{8\sqrt{\pi}}\left(\frac{f(x)}{x}-f'(x)\right)dx\\
&+& \int\limits_0^\infty \Bigg(2^{-\frac{\upsilon+3}{2}}\pi^{\frac{1-\upsilon}{2}}x^{-\frac{\upsilon+5}{2}}\frac{T^{-\frac{\upsilon+1}{7}}}{n^{\frac{6-\upsilon}{7}}}
C_3(\upsilon,x)^{1-\upsilon}dx\int\limits_1^\infty\alpha(\tau)^\upsilon d\tau, \nonumber
\end{eqnarray}
where
\begin{eqnarray*}T&=&\frac{3\int_0^\infty x^{-3/2}f(x)dx}{\sqrt{\pi}\int_{0}^\infty
\left(\frac{f(x)}{3x^2}+f''(x)\right)^2dx}.
\end{eqnarray*}
The last term in \eqref{12_1} has the rate $o(n^{\frac{\upsilon-6}{7}})$. By $0<\upsilon<1$ we get that the optimal rate of convergence of MISE is given by
$ MISE_{opt}(\hat f'(x)) = O(n^{-4/7})$.
\end{proof}
\begin{proof}[Proof of Lemma~\ref{lem3}]\label{ap3}
We have to prove that $\alpha(\tau)$ defined by \eqref{alpha} satisfies the conditions
of Lemma \ref{Lem1}. Conditions 2 and  3 of Lemma \ref{Lem1} only refer to the density distribution.
Thus, we remain  to
check only the first condition of Lemma \ref{Lem1}.
\par To this end, using \eqref{alpha} we get
\begin{eqnarray}\label{17} \int\limits_1^{\infty}\alpha(\tau)^\upsilon d\tau &\leq&
\int\limits_1^{\tau_0}d\tau+\int\limits_{\tau_0}^{\infty}\left(2(C+1)\mathsf E|X_i|^\nu|\rho^\nu|^\tau\right)^\upsilon d\tau\\
&=&\tau_0-1+\left(2(C+1)\mathsf E|X_i|^\nu\right)^\upsilon \int\limits_{\tau_0}^{\infty}\left(|\rho^\nu|^\tau\right)^\upsilon d\tau.\nonumber
\end{eqnarray}
The integral in \eqref{17} can be taken in general as
\begin{eqnarray*}&&\label{18} \int\limits_{\tau_0}^{\infty}\left(|\rho^\nu|^\tau\right)^\upsilon d\tau = \frac{|\rho^\nu|^{\tau \upsilon}}{\upsilon\ln(|\rho^\nu|)}\Big|_{\tau_0}^{\infty}
\end{eqnarray*}
Thus, to satisfy the first condition of Lemma \ref{Lem1}, it must be
\begin{eqnarray}&& \label{19}|\rho^\nu|^{\tau\upsilon}\Big|_{\tau=\infty}<\infty.
\end{eqnarray}
Since $\rho\in(-1,1)$ holds, it follows $|\rho|\in[0,1)$. For $\rho=0$ \eqref{19} is satisfied. For $|\rho|\in(0,1)$  one can rewrite \eqref{19} as
\begin{eqnarray*}&& \label{20}\left(\frac{1}{\xi}\right)^{\nu\tau\upsilon}\Big|_{\tau=\infty}<\infty, \quad \xi>1,
\end{eqnarray*}
which is valid as $\nu\upsilon>0$. The latter is true since $0<\upsilon<1$ and $\nu=\min\{p,q,1\}>0$.
Thus, the strong mixing AR(1) process \eqref{15}  satisfies Lemma \ref{Lem1}. Hence, it satisfies the conditions of Theorem \ref{thm}.
\end{proof}



\begin{thebibliography}{8}

\bibitem{Aksoy}
 \textsc{Aksoy, H. } (2000). Use of Gamma Distribution in Hydrological Analysis.
 \textit{Turk J. Engin Environ Sci}, \textbf{24}, 419 -- 428.

\bibitem{Donald:83}
\textsc{Andrews, D.W.K.} (1983). First order autoregressive processes and strong mixing.
\textit{Yale University}, New Haven, Connecticut.

\bibitem{Bosq:96}
 \textsc{Bosq, S.} (1996). Nonparametric Statistics for Stochastic Processes. Estimation and Prediction,
 \textit{Springer}, New York.

 \bibitem{TaufikBouezmarnia:Rom}
\textsc{Bouezmarnia, T. \ \textup{and}  Rombouts, J.V.K.} (2007). Nonparametric density estimation for multivariate bounded data.
\textit{Journal of Statistical Planning and Inference}, \textbf{140}, 1, 139–-152.

\bibitem{TaufikBouezmarnia:Rom2}
\textsc{Bouezmarnia, T. \ \textup{and}  Rombouts, J.V.K.} (2010). Nonparametric density estimation for positive times series.
\textit{Computational Statistics and  Data Analysis},
\textbf{54}, 2, 245–-261.

\bibitem{BouSca}
  \textsc{Bouezmarnia, T. and Scaillet, O.} (2003).
Consistency of Asymmetric Kernel Density
Estimators and Smoothed Histograms
with Application to Income Data.
   \textit{Econometric
Theory}, \textbf{21}, 390--412.

\bibitem{Bhattacharya}
  \textsc{Bhattacharya, P.K.} (1967).
 Estimation of a Probability Density Function and its Derivatives.
       \textit{The Indian Journal of Statistics},
   \textbf{A 29}, 373--382.

\bibitem{Chen:20}
\textsc{Song Xi Chen} (2000).
     Probability density function estimation using gamma kernels.
    \textit{Annals of the Institute of Statistical Mathematics}
      \textbf{54}, 471--480.

 \bibitem{Brabanter}
  \textsc{De Brabanter, K. \ \textup{and} De Brabanter, J.\ \textup{and} De Moor, B.} (2011).
  Nonparametric Derivative Estimation.
  \textit{Proc. of the 23rd Benelux Conference on Artificial Intelligence (BNAIC)}, Gent, Belgium,  75--81.

\bibitem{Dobrovidov:12}
\textsc{Dobrovidov, A.V. \ \textup{and} Koshkin, G.M.  \ \textup{and}  Vasiliev, V. A.} (2012).
Non-parametric state space models.
\textit{Kendrick press}, USA.

\bibitem{DobrovidovMarkovich:13a}
\textsc{Dobrovidov, A.V. \ \textup{and} Markovich, L.A.} (2013).
Nonparametric gamma kernel estimators of density derivatives on positive semi-axis.
\textit{Proc. of IFAC MIM 2013: Petersburg, Russia, June 19–21}, 944--949.

\bibitem{DobrovidovMarkovich:13b}
\textsc{Dobrovidov, A.V. \ \textup{and} Markovich, L.A.} (2013).
Data-driven bandwidth choice for gamma kernel estimates of density derivatives
 on the positive semi-axis.
 \textit{Proc. of IFAC International Workshop on Adaptation and Learning in Control and Signal Processing Caen, France}, 500--505.

\bibitem{Furman}
 \textsc{Furman, E.} (2008). On a multivariate Gamma distribution.
  \textit{Statist. Probab. Lett.}, \textbf{78}, 2353--2360.

\bibitem{HallWehrly}
  \textsc{Hall, P. \ \textup{and} Wehrly, T.E.} (1991).
A geometrical method for removing edge effects from kernel-type nonparametric regression estimators.
 \textit{J. Amer. Statist. Assoc.}, \textbf{86}, 665--672.

\bibitem{Hurlimann}
 \textsc{H\"{u}rlimann, W.} (2001). Analytical Evaluation of Economic Risk Capital for Portfolios of Gamma Risks.
 \textit{ASTIN Bulletin}, \textbf{31}, 107--122.

\bibitem{Metro}
  \textsc{Hastings, W.K.} (1970).
Monte Carlo Sampling Methods Using Markov Chains and Their Applications.
 \textit{Biometrika}, \textbf{57}, 1, 97--109.

\bibitem{Jones}
  \textsc{Jones, M.C.} (1993).
 Simple boundary correction for density estimation kernel.
 \textit{Statistics and Computing}, \textbf{3}, 135--146.

\bibitem{LejeuneSarda}
  \textsc{Lejeune, M. \ \textup{and}  Sarda, P. } (1992).
   Smooth Estimators of Distribution and Density Functions.
 \textit{Computational Statistics and Data Analysis}, \textbf{14}, 457–-471.

\bibitem{Muller}
  \textsc{M\"{u}ller,  H.G.} (1991).
   Smooth Optimum Kernel Estimators Near Endpoints.
    \textit{Biometrika}, \textbf{78}, 3, 521--530.

\bibitem{Parzen}
   \textsc{Parzen, E.} (1962).
 On Estimation of a Probability Density Function and Mode.
   \textit{The Annals of Mathematical Statistics}, \textbf{33}, 3, 1065.

\bibitem{Roberts}
   \textsc{Roberts, G.O. \ \textup{and}  Rosenthal, J.S. \ \textup{and} Segers, J. \ \textup{and} Sousa, B.}, (2007).
Extremal indices, geometric ergodicity of Markov chains, and MCMC.
  \textit{Extremes}, \textbf{9}, 3-4, 213--229.


\bibitem{Rosenblatt}
   \textsc{ Rosenblatt, M.} (1956).
Remarks on Some Nonparametric Estimates of a Density Function.
   \textit{The Annals of Mathematical Statistics}, \textbf{27}, 3, 832.

  \bibitem{Sasaki}
  \textsc{H. Sasaki, A. \ \textup{and}  Hyv\"{a}rinen, \ \textup{and} M. Sugiyama} (2014).
   Clustering via mode seeking by direct estimation of the gradient of a log-density.
     \textit{In Proceedings of the
European Conference on Machine Learning and Principles and Practice of
Knowledge Discovery in Databases (ECML/PKDD 2014)}, to appear.

 \bibitem{Scailet}
  \textsc{Scaillet, O.} (2004).
Density Estimation Using Inverse and Reciprocal Inverse
Gaussian Kernels.
 \textit{Journal of Nonparametric Statistics},
 \textbf{16}, 217--226.

\bibitem{Schuster}
     \textsc{Schuster, E.F.} (1985)
    Incorporating support constraints into nonparametric estimators of densities.
    \textit{Commun. Statist. Theory Methods}, \textbf{14}, 1123--1136.

\bibitem{Schuster2}
    \textsc{Schuster, E.F.} (1969)
    Estimation of a probability function and its derivatives.
     \textit{Ann. Math. Statist.}, \textbf{40}, 1187--1195.

\bibitem{Skold}  \textsc{Sk\"{o}ld, M. \ \textup{and} Roberts, G.O.} (2003). Density estimates from the Metropolis-Hastings Algorithm.
 \textit{Scand. J. Stat.}, \textbf{30}, 699--718.
\bibitem{Tsypkin:85}
\textsc{Tsypkin, Ya. Z.} (1985). Optimality in adaptive control systems. Uncertainty and Control.
\textit{Springer, Lecture Notes in Control and Information Sciences Berlin, Heidelberg}, \textbf{70}, 153--214.

\bibitem{Turlach}
\textsc{Turlach, B.A.} (1993). Bandwidth Selection in Kernel Density Estimation: A Review.
\textit{CORE and Institut de Statistique}.

 \bibitem{WandJones}
 \textsc{Wand, M.P. \ \textup{and} Jones, M.C.} (1995). Kernel Smoothing.
 \textit{Chapman and Hall}, London.

\bibitem{Zhang}
\textsc{Zhang, S.} (2010). A note on the performance of the gamma kernel estimators at the boundary.
 \textit{Statis. Probab. Lett.}, \textbf{80}, 548--557.

\end{thebibliography}
\end{document}